\documentclass[12pt,reqno]{amsart}
\usepackage{amsmath}
\usepackage{amssymb}
\usepackage{amstext}
\usepackage{mathrsfs}
\usepackage{a4wide}
\usepackage{graphicx}
\usepackage{bbm}
\allowdisplaybreaks \numberwithin{equation}{section}
\usepackage{color}
\usepackage{cases}

\numberwithin{equation}{section}

\newtheorem{theorem}{Theorem}[section]

\newtheorem{lemma}[theorem]{Lemma}

\theoremstyle{definition}

\theoremstyle{remark}
\newtheorem{remark}[theorem]{Remark}

\begin{document}

\title
{On the global classical solutions for the generalized SQG equation}

\author{Daomin Cao, Guolin Qin,  Weicheng Zhan, Changjun Zou}

\address{Institute of Applied Mathematics, Chinese Academy of Sciences, Beijing 100190, and University of Chinese Academy of Sciences, Beijing 100049,  P.R. China}
\email{dmcao@amt.ac.cn}
\address{Institute of Applied Mathematics, Chinese Academy of Sciences, Beijing 100190, and University of Chinese Academy of Sciences, Beijing 100049,  P.R. China}
\email{qinguolin18@mails.ucas.edu.cn}
\address{Institute of Applied Mathematics, Chinese Academy of Sciences, Beijing 100190, and University of Chinese Academy of Sciences, Beijing 100049,  P.R. China}
\email{zhanweicheng16@mails.ucas.ac.cn}

\address{Institute of Applied Mathematics, Chinese Academy of Sciences, Beijing 100190, and University of Chinese Academy of Sciences, Beijing 100049,  P.R. China}
\email{zouchangjun17@mails.ucas.ac.cn}


\begin{abstract}
In this paper, we study the existence of global classical solutions to the generalized surface quasi-geostrophic equation. By using the variational method, we provide some new families of global classical solutions for to the generalized surface quasi-geostrophic equation. These solutions mainly consist of rotating solutions and travelling-wave solutions.
\end{abstract}

\maketitle

\section{Introduction and main results}
In this paper, we consider the generalized surface quasi-geostrophic (gSQG) equation
\begin{align}\label{1-1}
\begin{cases}
\partial_t\vartheta+\mathbf{v}\cdot \nabla \vartheta =0&\text{in}\ \mathbb{R}^2\times (0,T)\\
 \ \mathbf{v}=\nabla^\perp(-\Delta)^{-s}\vartheta     &\text{in}\ \mathbb{R}^2\times (0,T),\\
\end{cases}
\end{align}
where $0< s<1$, $\vartheta(x,t):\mathbb{R}^2\times (0,T)\to \mathbb{R}$ is the active scaler being transported by the velocity field $\mathbf{v}(x,t):\mathbb{R}^2\times (0,T)\to \mathbb{R}^2$ generated by $\vartheta$, and $(a_1,a_2)^\perp=(a_2,-a_1)$.

When $s={1}/{2}$, \eqref{1-1} is the inviscid surface quasi-geostrophic (SQG) equation, which models the evolution of the temperature from a general quasi-geostrophic system for atmospheric and atmospheric flows; see \cite{Con, He, La}. In the particular case $s\uparrow1$ we obtain the vorticity formulation of the two-dimensional incompressible Euler equation; see \cite{MB}. The limiting case $s\downarrow0$ produces stationary solutions. \eqref{1-1} has a strong mathematical and physical analogy with the three-dimensional incompressible Euler equation; see \cite{Con} for more details. For this reason, in recent years, there has been tremendous interest in the gSQG equation.

The purpose of this paper is to construct nontrivial global (classical) solutions of the gSQG equation. It is well known that all radially symmetric functions $\vartheta$ are stationary solutions to the gSQG equation due to the structure of the nonlinear term. An interesting issue is the existence of other global classical solutions. We would like to mention that the Cauchy problem for the gSQG equation is extremely delicate. To our knowledge, the problem of whether the gSQG system presents finite time singularities or there is global well-posedness of classical solutions is still open (see, e.g., \cite{Cas2, Kis1}).

In \cite{Cas2}, Castro et al. provided a first construction of nontrivial global classical solutions of the SQG equation by developing a bifurcation argument from a specific radially symmetric function. The classical solutions constructed in \cite{Cas2} are uniformly rotating solutions that evolve by rotating with constant angular velocity around its center of mass. In \cite{Gra}, Gravejat and Smets put forward an alternative way of constructing  smooth families of special global solutions. They constructed families of smooth travelling-wave solutions to the SQG equation by using the variational method. This result was generalized by Godard-Cadillac \cite{Go0} to the gSQG equation. Recently, Ao et al. \cite{Ao} successfully applied the Lyapunov-Schmidt reduction method to construct rotating and travelling-wave smooth solutions to the gSQG equation. Besides smooth global solutions, there have also been extensive study of uniformly rotating patches (also known as V-states); see, e.g., \cite{Cas1, Has0, Has} and the references therein. As for weak solutions of the gSQG equation, we refer the reader to \cite{Buc, Mar, Res} for some relevant results.

In this paper, we provide some new families of smooth global solutions for the gSQG equation. Before stating our main results, we first explicit the equations satisfied by rotating solutions. We look for rotating solutions $\vartheta$ to \eqref{1-1} under the form
\begin{equation}\label{1-2}
	\vartheta(x,t)=\omega(Q_{-\alpha t}x),
\end{equation}	
where $\omega$ is some profile function defined on $\mathbb{R}^2$, and $Q_{\phi}$ stands for the counterclockwise rotation of angle $\phi$. Recall that the operator $(-\Delta)^{-s}$ is given by the expression
\begin{equation*}
	(-\Delta)^{-s}\omega(x)=\mathcal{G}_s\omega(x)=\int_{\mathbb{R}^2}G_s(x-y)\omega(y)dy,
\end{equation*}
where $G_s$ is the fundamental solution of $(-\Delta)^{s}$ in $\mathbb{R}^2$ given by
\begin{equation*}
	G_s(z)=\frac{c_s}{|z|^{2-2s}},\ \ \ \ c_s=\frac{\Gamma(1-s)}{2^{2s}\pi\Gamma(s)}.
\end{equation*}
If we set
\begin{equation*}
	\mathbf{u}=\nabla^\perp\mathcal{G}_s\omega,
\end{equation*}
then by \eqref{1-1} and \eqref{1-2}, we can recover the velocity field
\begin{equation*}
	\mathbf{v}(x,t)=Q_{\alpha t}\mathbf{u}(Q_{-\alpha t}x)
\end{equation*}
and the first equation in \eqref{1-1} is reduced to a stationary equation
\begin{equation}\label{1-8}
\nabla^\perp(\mathcal{G}_s\omega+\frac{\alpha}{2} |x|^2)\cdot\nabla\omega=0.
\end{equation}
Hence it is natural to introduce the weak formulation of \eqref{1-8}, namely, $\omega$ satisfies
\begin{equation}\label{1-9}
	\int_{\mathbb{R}^2}\omega\nabla^\perp(\mathcal{G}_s\omega+\frac{\alpha}{2} |x|^2)\cdot\nabla \varphi  dx=0, \ \ \ \forall\,\varphi\in C_0^\infty(\mathbb{R}^2).
\end{equation}

The vortex solutions we will construct belong to $L^\infty(\mathbb{R}^2)$ and they are also of compact support. Due to the standard elliptic theory for Riesz potentials the corresponding functions $\mathcal{G}_s\omega\in \dot{W}^{2s,p}(\mathbb{R}^2)$ for any $1<p<\infty$, so the integral in \eqref{1-9} makes sense. Notice that if $s<1/2$, the regularity is not sufficient to provide a rigorous meaning to \eqref{1-9}. For this reason, we shall restrict our attention to $s\in[1/2, 1)$.

As remarked by Arnol'd \cite{Ar1}, a natural way of obtaining solutions to the stationary problem \eqref{1-8} is to impose that $\omega$ and $\mathcal{G}_s\omega+{\alpha} |x|^2/2$ are (locally) functional dependent. More precisely, we may impose that
\begin{equation*}
  \omega=f(\mathcal{G}_s\omega+{\alpha} |x|^2/2)
\end{equation*}
for some Borel measurable function $f:\mathbb{R}\to \mathbb{R}$. For technical reasons, we assume that
\begin{itemize}
\item [($\text{A}$)]\ $f(\tau)=0\ \ \text{for} \ \tau\le 0;\ \ f(\tau)>0\ \ \text{for} \ \tau>0.$
\end{itemize}

We are now in a position to state our main results. Our first main result establishes the existence of co-rotating vortices with $N$-fold symmetry for the gSQG equation. More precisely, we have the following theorem:
\begin{theorem}\label{thm1}
Let $1/2\le s<1$. Let $N\ge2$ be a integer. Suppose that $f$ is a bounded nondecreasing function satisfying (A). Then there exists $\varepsilon_0>0$ such that for any $\varepsilon\in (0,\varepsilon_0]$, \eqref{1-1} has a global rotating solution $\vartheta_\varepsilon (x,t)$ with the following properties:
\begin{itemize}
  \item[(i)]$\vartheta_\varepsilon (x,t)=\omega_{ro, \varepsilon}(Q_{-\alpha_\varepsilon t}x)$, where the angular velocity $\alpha_\varepsilon \in \mathbb{R}$ and $\omega_{ro, \varepsilon}\in L^\infty(\mathbb{R}^2)$ is a weak solution to \eqref{1-8} in the sense of \eqref{1-9}.
  \item[(ii)]$\omega_{ro, \varepsilon}$ is $N$-fold symmetric, namely
      \begin{equation*}
      \omega_{ro, \varepsilon}(x)=\omega_{ro, \varepsilon}\left(Q_{\frac{2\pi}{N}}x\right)\,\,\,\text{for any}\, x\in \mathbb{R}^2.
      \end{equation*}

  \item[(iii)]There holds
       \begin{equation*}
    \omega_{ro, \varepsilon}=\frac{1}{\varepsilon^2} f\circ\psi_{ro, \varepsilon},
       \end{equation*}
   where
      \begin{equation*}
			\psi_{ro, \varepsilon}(x)=\mathcal{G}_s\omega_{ro, \varepsilon}+\frac{\alpha_\varepsilon}{2}|x|^2-\mu_{ro,\varepsilon}
	  \end{equation*}
		for some $\mu_{ro,\varepsilon}\in \mathbb{R}$.
  \item[(iv)] One has,in the sense of measure
      \begin{equation*}
   \omega_{ro, \varepsilon}(x)\rightharpoonup \sum_{k=0}^{N-1}\pmb{\delta}_{Q_{\frac{2k\pi}{N}}(1,0)}\ \ \text{as}\ \ \varepsilon\to 0^+,
 \end{equation*}
 where $\pmb{\delta}(x)$ denotes the standard Dirac mass at the origin. Moreover, there exists a constant $\Lambda_0>0$ independent of $\varepsilon$ such that
 \begin{equation*}
   \text{supp}(\omega_{ro, \varepsilon})\subset \bigcup^{N-1}_{k=0}B_{\Lambda_0\varepsilon}\left(Q_{\frac{2k\pi}{N}}(1,0)\right).
 \end{equation*}
 \item[(v)] As $\varepsilon \to 0^+$, it holds
  \begin{equation}\label{angu}
    \alpha_\varepsilon\to \sum\limits_{k=1}^{N-1}\frac{c_s(1-s)}{|(1,0)-Q_{\frac{2k\pi}{N}}(1,0)|^{2-2s}}.
  \end{equation}
  In addition, concerning $\mu_{ro,\varepsilon}$, we have
    \begin{equation*}
      0<\liminf_{\varepsilon\to0^+}\varepsilon^{2-2s}\mu_{ro,\varepsilon}\le \limsup_{\varepsilon\to0^+}\varepsilon^{2-2s}\mu_{ro,\varepsilon}<+\infty.
    \end{equation*}
\end{itemize}
\end{theorem}

\begin{remark}
  We mention that the nonlinearity $f$ is allowed to be discontinuous. Because of the monotonicity assumption, it can only have at most countable points of jump discontinuities. If $f$ is sufficiently smooth, then according to standard elliptic theory, the solutions in Theorem \ref{thm1} are actually global classical solutions of the gSQG equation. Ao et al. \cite{Ao} considered the case when $f(\tau)=\tau_+^p$ with $1<p<\frac{1+s}{1-s}$. Their construction seems to rely in essential way on the $p$-power form of $f$, for example, it relies on the existence, uniqueness and asymptotic behavior of the ground state solution to the fractional plasma equation established by Chan et al. \cite{Chan}. We construct the solutions by the variational method, which does not rely on the properties of the limit equation (see section \ref{Sec2} below). If $f$ is a step function, then $\vartheta_\varepsilon$ is a uniformly rotating patch with $N$-fold symmetry. In \cite{HM}, Hmidi and Mateu established the existence of co-rotating and counter-rotating vortex pairs of simply connected patches for the gSQG equation by using the contour dynamics equations. Later, co-rotating vortex patches with $N$-fold symmetry for the gSQG equation was obtained by Garc\'ia \cite{Gar2} and Godard-Cadillac et al. \cite{Go}. However, in the absence of comprehensive uniqueness theory, the correspondences between the solutions constructed by the various methods remains unclear.
\end{remark}

\begin{remark}
  In the limit $\varepsilon \to 0^+$, we obtain a desingularization of $N$ point vortices located at the vertex of a regular polygon with $N$ sides (also called Thomson polygon). The study of equal point vortices located at the vertices of a polygon, which rotates around its center, can be traced back to the work of Lord Kelvin 1878 and Thomson 1883 (see, e.g., \cite{Kur, New}). We remark that the limiting angular velocity in \eqref{angu} does correspond exactly to the speed of rotation of $N$ point vortices located at the vertex of a regular polygon with $N$ sides evolving according to the gSQG equation, see \cite{Gar2}.
\end{remark}

Our second main result concerns the existence of travelling-wave solutions. Before stating the result, let us derive the equation satisfied by travelling-wave solutions. Up to a rotation, we may assume, without loss of generality, that these waves have a negative speed $-W$ in the vertical direction, so that
 \begin{equation*}
   \vartheta(x,t)=\omega(x_1, x_2+Wt).
 \end{equation*}
In this setting, the first equation in \eqref{1-1} is also reduced to a stationary equation
\begin{equation}\label{t1-1}
\nabla^\perp(\mathcal{G}_s\omega-Wx_1)\cdot\nabla\omega=0,
\end{equation}
which has a weak form
\begin{equation}\label{t1-2}
	\int_{\mathbb{R}^2}\omega\nabla^\perp(\mathcal{G}_s\omega-Wx_1)\cdot\nabla \varphi dx=0, \ \ \ \forall\,\varphi\in C_0^\infty(\mathbb{R}^2).
\end{equation}
Similarly, a natural way of obtaining solutions to the stationary
problem \eqref{t1-1} is to impose that $\omega$ and $\mathcal{G}_s\omega-Wx_1$ are (locally) functional dependent, namely,
\begin{equation*}
  \omega=f(\mathcal{G}_s\omega-Wx_1).
\end{equation*}
\begin{theorem}\label{thm2}
Let $1/2\le s<1$. Let $W>0$ be given. Suppose $f$ is a bounded nondecreasing function satisfying (A). Then there exists $\varepsilon_0>0$ such that for any $\varepsilon\in (0,\varepsilon_0]$, \eqref{1-1} has a global travelling-wave solution $\vartheta_\varepsilon (x,t)$ with the following properties:
\begin{itemize}
  \item[(i)]$\vartheta_\varepsilon (x,t)=\omega_{tr,\varepsilon}(x_1,x_2+Wt)$, where $\omega_{tr,\varepsilon}\in L^\infty(\mathbb{R}^2)$ is a weak solution to \eqref{t1-1} in the sense of \eqref{t1-2}.
  \item[(ii)]$\omega_{tr,\varepsilon}$ is an odd function with respect to the variable $x_1$. That is,
  \begin{equation*}
    \omega_{tr,\varepsilon}(x_1,x_2)=-\omega_{tr,\varepsilon}(-x_1,x_2),\ \ \ \forall\,x\in \mathbb{R}^2.
  \end{equation*}
  \item[(iii)]There holds
       \begin{equation*}
    \omega_{tr, \varepsilon}=\frac{1}{\varepsilon^2} f\circ\psi_{tr,\varepsilon},
       \end{equation*}
   where
      \begin{equation*}
			\psi_{tr, \varepsilon}(x)=\mathcal{G}_s\omega_{tr, \varepsilon}-Wx_1-\mu_{tr,\varepsilon}
	  \end{equation*}
		for some $\mu_{tr,\varepsilon}\in \mathbb{R}$.
  \item[(iv)] One has, in the sense of measure
		\begin{equation*}
   \omega_{\text{tr}, \varepsilon}(x)\rightharpoonup \pmb{\delta}(x-b_1)-\pmb{\delta}(x-b_2)\ \ \ \text{as}\ \ \varepsilon\to 0^+,
 \end{equation*}
 where
		\begin{equation*}
			d=\left(\frac{1}{4\pi W}\frac{\Gamma(2-s)}{\Gamma(s)}\right)^{\frac{1}{3-2s}}, \ b_1=d\mathbf{e}_1,\ b_2=-d\mathbf{e}_1,\ \mathbf{e}_1=(1,0).
		\end{equation*}
Moreover, there exists a constant $\Lambda_1>0$ independent of $\varepsilon$ such that
 \begin{equation*}
   \text{diam}\left(\text{supp}(\omega_{tr,\varepsilon})\cap \mathbb{R}^2_+  \right)\le \Lambda_1\varepsilon.
 \end{equation*}
 \item[(v)]In addition, concerning $\mu_{tr,\varepsilon}$, we have
    \begin{equation*}
      0<\liminf_{\varepsilon\to0^+}\varepsilon^{2-2s}\mu_{tr,\varepsilon}\le \limsup_{\varepsilon\to0^+}\varepsilon^{2-2s}\mu_{tr,\varepsilon}<+\infty.
    \end{equation*}
\end{itemize}
\end{theorem}

\begin{remark}
  As mentioned earlier, smooth travelling-wave solutions to the gSQG equation were constructed in \cite{Ao, Go0, Gra}. However, we would like to point out that in all of these works the nonlinearity $f$ is assumed to be unbounded. Theorem \ref{thm2} can be taken as a complement to these works.
\end{remark}

\begin{remark}
  In the limit $\varepsilon\to 0^+$, we obtain a desingularization of a pair of point vortices with equal magnitude and opposite signs. In addition, we see that two point vortices with equal magnitude and opposite signs at distance $2d$ exhibit a uniform translating motion with the speed $$W=\frac{\Gamma(2-s)}{4\pi \Gamma(s) d^{3-2s}}.$$ This classical result is well-known in the literature, see \cite{ Ao, March, Ros} for example.
\end{remark}

In the next section, we provide a variational construction of co-rotating vortices with $N$-fold symmetry for the gSQG equation. The proof of Theorem \ref{thm1} is also given in this section. In section \ref{Sec3}, we consider translating vortex pairs for the
gSQG equation and give the proof of Theorem \ref{thm2}.

\section{Construction of co-rotating vortices for the gSQG equation}\label{Sec2}
In this section, we provide a variational construction of co-rotating vortices with $N$-fold symmetry for the gSQG equation. To begin with, we need to introduce some notations. We denote by $B_R(x)$ the open ball in $\mathbb{R}^2$ of center $x$ and radius $R>0$. If $\Omega\subset\mathbb{R}^N$ is measurable then $\text{meas}\,({\Omega})$ denotes the $N$-dimensional Lebesgue measure of $\Omega$. $\textbf{1}_\Omega$ denotes the characteristic function of $\Omega$. The right half-plane is denoted by $\mathbb{R}^2_+:=\{(x_1,x_2)\in \mathbb{R}^2: x_1>0\}$. In the usual polar coordinates $(r,\theta)$, we shall say that nonnegative $\xi\in L^1(\mathbb{R}^2)$ is Steiner symmetric with respect to $\theta$ if $\xi$ for the variable $\theta$ is the unique even function such that
\begin{equation*}
  \xi(r,\theta)>\tau\ \ \ \text{if and only if}\ \ \ |\theta|<\frac{1}{2}\,\text{meas}\left\{\theta'\in (-\pi,\pi):\xi(r,\theta')>\tau \right\},
\end{equation*}
for any positive numbers $r$ and $\tau$, and any $-\pi<\theta<\pi$. We denote by $\text{supp}(\xi)$ the (essential) support of $\xi$. Set
\begin{equation*}
  \mathcal{U}_N:=\Big\{(r\cos\theta,r\sin\theta)\in \mathbb{R}^2:-\frac{\pi}{N}<\theta<\frac{\pi}{N} \Big\}.
\end{equation*}

\subsection{Variational problem}\label{s0}
Consider the kinetic energy of the fluid
\begin{equation*}
  {\text{KE}}_s(\omega):=\frac{1}{2}\int_{\mathbb{R}^2}\int_{\mathbb{R}^2}G_s(x-x')\omega(x)\omega(x')dxdx',
\end{equation*}
and its angular momentum
\begin{equation*}
  L(\omega)=\int_{\mathbb{R}^2}|x|^2\omega(x)dx.
\end{equation*}
The total vorticity is defined by
\begin{equation*}
  M(\omega)=\int_{\mathbb{R}^2}\omega(x)dx.
\end{equation*}

These quantities are conserved for sufficiently regular solutions to \eqref{1-1} (see for example \cite{Buc}). Thanks to the $N$-fold symmetry of the desired solution, the kinetic energy may be rewritten as
\begin{equation*}
  {\text{KE}}_s(\omega)=\frac{N}{2}\int_{\mathcal{U}_N}\int_{\mathcal{U}_N}K_s(x,x')\omega(x)\omega(x')dxdx',
\end{equation*}
where the kernel $K_s$ is given by
\begin{equation*}
  K_s(x,x')=\sum_{k=0}^{N-1}G_s\left(x-Q_{\frac{2k\pi}{N}}x'\right).
\end{equation*}
For more information about kernel $K_s$, we refer to \cite{Go} (see also the appendix below). Based on this observation, we shall restrict the construction to only one vortex inside the angular sector $\mathcal{U}_N$.

Let $f:\mathbb{R}\to \mathbb{R}$ be a bounded nondecreasing function satisfying (A). Let $F(\tau)=\int_{0}^{\tau}f(\tau')d\tau'$. Let $J$ be the conjugate function to $F$ defined by $J(s)=\sup_{t\in \mathbb{R}}\left[st-F(t)\right]$ (see \cite{Roc}). Without loss of generality, hereafter we may assume that $\sup_{\mathbb{R}}f=1$.

We introduce the energy functional
\begin{equation*}
  \mathcal{E}_\varepsilon(\omega)=\frac{1}{2}\int_{\mathcal{U}_N}\int_{\mathcal{U}_N}K_s(x,x')\omega(x)\omega(x')dxdx'-\frac{1}{\varepsilon^2}\int_{\mathcal{U}_N} J(\varepsilon^2\omega(x))dx,
\end{equation*}
and the function
\begin{equation*}
  \mathcal{K}_s\omega(x)=\int_{\mathcal{U}_N}K_s(x,x')\omega(x')dx'.
\end{equation*}
Let
\begin{equation*}
	S:=\left\{(r\cos\theta,r\sin\theta)\in \mathbb{R}^2:\frac{1}{2}\le r\le\frac{3}{2},\ \ -\frac{\pi}{2N}\le\theta\le\frac{\pi}{2N} \right\},
\end{equation*}
and
\begin{equation*}
  \mathcal{A}:=\Big{\{}\omega\in L^{1+\frac{1}{s}}(\mathbb{R}^2): \omega\ge0,\ \text{supp}(\omega)\subseteq S,\ M(\omega)=L(\omega)=1 \Big{\}}.
\end{equation*}
We shall consider the maximization problems:
\begin{equation}\label{maxi}
\mathcal{C}_{\varepsilon}:=\sup\Big{\{} \mathcal{E}_{\varepsilon}(\omega): \omega\in \mathcal{A} \Big{\}}.
\end{equation}
However, such maximization problems do seem difficult to deal with directly. Instead, we will consider the following penalized problems.
\subsection{The penalized problems}\label{s1}
Let
\begin{equation*}
  f_\lambda(\tau)=f(\tau)+\lambda\tau_+^s,\ \ \ \ \ \  F_\lambda(\tau)=\int_{0}^{\tau}f_\lambda(\tau')d\tau'.
\end{equation*}
Let $J_\lambda$ denote the conjugate function to $F_\lambda$. Then
\begin{equation*}
  \lambda \tau_+^s\le f_\lambda(\tau)\le 1+\lambda\tau_+^s.
\end{equation*}
Hence
\begin{equation}\label{2-1}
  \frac{(\tau-1)_+^{1+\frac{1}{s}}}{(1+\frac{1}{s})\lambda^\frac{1}{s}}\le J_\lambda(\tau)\le \frac{\tau_+^{1+\frac{1}{s}}}{(1+\frac{1}{s})\lambda^\frac{1}{s}}.
\end{equation}
Let
\begin{equation*}
  \mathcal{E}_{\varepsilon,\lambda}(\omega)=\frac{1}{2}\int_{\mathcal{U}_N}\int_{\mathcal{U}_N}K_s(x,x')\omega(x)\omega(x')dxdx'-\frac{1}{\varepsilon^2}\int_{\mathcal{U}_N} J_\lambda(\varepsilon^2\omega(x))dx.
\end{equation*}
For convenience, we denote
\begin{equation*}
\begin{split}
    & {E}_s(\omega)=\frac{N}{2}\int_{\mathcal{U}_N}\int_{\mathcal{U}_N}K_s(x,x')\omega(x)\omega(x')dxdx',\ \ \ \\
     & \mathcal{J}_\varepsilon(\omega)=\frac{1}{\varepsilon^2}\int_{\mathbb{R}^2} J(\varepsilon^2\omega(x))dx,\ \ \ \\
     &\mathcal{J}_{\varepsilon, \lambda}(\omega)=\frac{1}{\varepsilon^2}\int_{\mathbb{R}^2} J_\lambda(\varepsilon^2\omega(x))dx.
\end{split}
\end{equation*}
Notice that $K_s(\cdot,\cdot)\in L^q(S\times S)$ for any $1\le q <1/(1-s)$. Hence ${E}_{s}$ is well-defined on $\mathcal{A}$.
Set
\begin{equation*}
  \varpi_\varepsilon=\frac{1}{\varepsilon^2}\textbf{1}_{B_{\varepsilon/\sqrt{\pi}}(a_\varepsilon, 0)},\ \ \text{with}\ \ a_\varepsilon=\sqrt{1-{\varepsilon^2}/{2\pi}}.
\end{equation*}
It is easy to verify that $\varpi_\varepsilon \in \mathcal{A}$ if $0<\varepsilon\le \rho_0:=\sin(\pi/2N)/6\sqrt{\pi}$.
\begin{lemma}\label{lem2-1}
Let $0< s<1$, $0<\varepsilon\le\rho_0$ and $\lambda>0$. Then there exists an $\omega_{\varepsilon,\lambda}\in \mathcal{A}$, which is Steiner symmetric with respect to $\theta$, such that
\begin{equation*}
  \mathcal{E}_{\varepsilon,\lambda}(\omega_{\varepsilon,\lambda})=\sup_{\omega\in \mathcal{A}}\mathcal{E}_{\varepsilon,\lambda}(\omega)<+\infty.
\end{equation*}
\end{lemma}
\begin{proof}
For any $\omega\in \mathcal{A}$, by the definition of $K_s$ and H\"older inequality, one verify that there exists a positive constant $C_s$, depending only on $s$, such that
\begin{equation*}
  \|\mathcal{K}_s\omega\|_{L^\infty(\mathbb{R}^2)}\le C_s\|\omega\|_{L^{1+\frac{1}{s}}(S)}.
\end{equation*}
Thus
\begin{equation*}
  E_s(\omega)\le C_s\|\omega\|_{L^1(S)}\|\omega\|_{L^{1+\frac{1}{s}}(S)}=C_s\|\omega\|_{L^{1+\frac{1}{s}}(S)}.
\end{equation*}
Using Young inequality, we have
\begin{equation}\label{2-2}
     E_s(\omega)
        \le C_{s,\varepsilon}\lambda+\frac{\varepsilon^\frac{2}{s}\|\omega\|_{L^{1+\frac{1}{s}}(S)}^{1+\frac{1}{s}}}{4\left(1+\frac{1}{s}\right)\lambda^\frac{1}{s}},
\end{equation}
for some positive constant $C_{s,\varepsilon}$ depending only on $s,\varepsilon$. On the other hand, by \eqref{2-1} we have
\begin{equation}\label{2-3}
\mathcal{J}_{\varepsilon,\lambda}(\omega)\ge \frac{1}{\varepsilon^2}\int_S \frac{(\varepsilon^2\omega-1)_+^{1+\frac{1}{s}}}{\left(1+\frac{1}{s}\right)\lambda^\frac{1}{s}}dx\ge \frac{\varepsilon^\frac{2}{s}\|\omega\|_{L^{1+\frac{1}{s}}(S)}^{1+\frac{1}{s}}}{2\left(1+\frac{1}{s}\right)\lambda^\frac{1}{s}}-\frac{C_{s,\varepsilon}}{\lambda^\frac{1}{s}}.
\end{equation}
Combining \eqref{2-2} and \eqref{2-3}, we get
\begin{equation*}
  \mathcal{E}_{\varepsilon,\lambda}(\omega)\le C_{s,\varepsilon,\lambda},\ \ \ \forall\,\omega\in \mathcal{A}.
\end{equation*}
Hence $\mathcal{E}_{\varepsilon,\lambda}$ is bounded from above on $\mathcal{A}$. Now, we turn to prove the existence of maximizers for $\mathcal{E}_{\varepsilon,\lambda}$ relative to $\mathcal{A}$. By \eqref{2-2} and \eqref{2-3}, we can further deduce that
\begin{equation*}
 \mathcal{E}_{\varepsilon,\lambda}(\omega)\le - \frac{\varepsilon^\frac{2}{s}\|\omega\|_{L^{1+\frac{1}{s}}(S)}^{1+\frac{1}{s}}}{4\left(1+\frac{1}{s}\right)\lambda^\frac{1}{s}}+ C_{s,\varepsilon}\lambda+\frac{C_{s,\varepsilon,\Lambda}}{\lambda^\frac{1}{s}}.
\end{equation*}
By \eqref{2-1}, it is easy to see that
\begin{equation*}
  \mathcal{E}_{\varepsilon,\lambda}(\varpi_\varepsilon)\ge C_{s,\varepsilon}-\frac{\varepsilon^\frac{2}{s}}{(1+\frac{1}{s})\lambda^\frac{1}{s}}\int_S \varpi_\varepsilon^{1+\frac{1}{s}}dx\ge C_{s,\varepsilon}-\frac{C_{s,\varepsilon}}{\lambda^\frac{1}{s}}.
\end{equation*}
Thus, we have
\begin{equation}\label{2-4}
  \|\omega\|_{L^{1+\frac{1}{s}}(S)}\le C_{s,\varepsilon},
\end{equation}
provided $\mathcal{E}_{\varepsilon,\lambda}(\varpi_\varepsilon)\le \mathcal{E}_{\varepsilon,\lambda}(\omega)$ and $0<\lambda\le1$. Therefore there exists a maximizing sequence $\{\omega_j\}_{j\in \mathbb{N}}\subset \mathcal{A}$, which is bounded in $L^{1+\frac{1}{s}}(S)$. Hence, there exists a function $\omega_{\varepsilon,\lambda}\in \mathcal{A}$ such that, up to a subsequence, $\{\omega_j\}_{j\in \mathbb{N}}$ weakly tends to $\omega_{\varepsilon,\lambda}$ in $L^{1+\frac{1}{s}}(S)$. Since $K_s(\cdot,\cdot)\in L^{1+s}(S\times S)$, we have
\begin{equation*}
  \lim_{j\to+\infty}E_s(\omega_j)=E_s(\omega_{\varepsilon,\lambda}).
\end{equation*}
On the other hand, by the lower semi-continuity, we have
\begin{equation*}
  \liminf_{j\to+\infty}\mathcal{J}_{\varepsilon,\lambda}(\omega_j)\ge \mathcal{J}_{\varepsilon,\lambda}(\omega_{\varepsilon,\lambda}).
\end{equation*}
Hence
\begin{equation*}
  \sup_{ \mathcal{A}}\mathcal{E}_{\varepsilon,\lambda}=\limsup_{j\to+\infty}\mathcal{E}_{\varepsilon,\lambda}(\omega_j)\le \mathcal{E}_{\varepsilon,\lambda}(\omega_{\varepsilon,\lambda})\le \sup_{ \mathcal{A}}\mathcal{E}_{\varepsilon,\lambda}.
\end{equation*}
This implies that $\omega_{\varepsilon,\lambda}$ is a maximizer. Moreover, by Lemma \ref{A2} in the appendix, we may further assume that $\omega_{\varepsilon,\lambda}$ is Steiner symmetric with respect to $\theta$. Indeed, if otherwise, then we can replace $\omega_{\varepsilon,\lambda}$ with its own angular Steiner symmetrization, which is still a maximizer. The proof is thus completed.
\end{proof}

\begin{lemma}\label{lem2-2}
There exist two numbers $\alpha_{\varepsilon,\lambda}$ and $\mu_{\varepsilon,\lambda}$ such that
\begin{equation}\label{2-5}
		\omega_{\varepsilon,\lambda}(x)=\frac{1}{\varepsilon^2} f_\lambda(\psi_{\varepsilon,\lambda}(x)),\ \ \text{a.e.}\ x\in S,
\end{equation}
with
\begin{equation*}
  \psi_{\varepsilon,\lambda}(x)=\mathcal{K}_s\omega_{\varepsilon,\lambda}(x)+\frac{\alpha_{\varepsilon,\lambda}}{2}|x|^2-\mu_{\varepsilon,\lambda}.
\end{equation*}
Moreover, $\alpha_{\varepsilon,\lambda}$ and $\mu_{\varepsilon,\lambda}$ are uniquely determined by $\omega_{\varepsilon,\lambda}$.
\end{lemma}

\begin{proof}
Fix a small positive number $\delta$ such that $D_\delta=\{x\in S: \omega_{\varepsilon,\lambda}(x)\ge \delta \}$ has positive measure. Then there exist two functions
$\phi_1$ and $\phi_2$ in $L^\infty(D_\delta)$ with $M(\phi_1)=L(\phi_2)=1$ and $M(\phi_2)=L(\phi_1)=0$. Let $\zeta\in L^\infty(S)$ be such that $\zeta\ge 0$ on $S\backslash D_\delta$. Consider the test function
\begin{equation*}
  \omega_\tau=\omega_{\varepsilon,\lambda}+\tau\left[\zeta-M(\zeta)\phi_1-L(\zeta)\phi_2 \right],\ \ \ \tau>0.
\end{equation*}
We check that $\omega_\tau\in \mathcal{A}$ for $\tau$ small. Since $\omega_{\varepsilon,\lambda}$ is a maximizer, we have
\begin{equation*}
		0\ge \frac{d}{d\tau}\bigg|_{\tau=0^+}\mathcal{E}_{\varepsilon,\lambda}(\omega_\tau)=\mathcal{E}_{\varepsilon,\lambda}'(\omega_{\varepsilon,\lambda})\zeta-M(\zeta)\mathcal{E}_{\varepsilon,\lambda}'(\omega_{\varepsilon,\lambda})\phi_1-L(\zeta)\mathcal{E}_{\varepsilon,\lambda}'(\omega_{\varepsilon,\lambda})\phi_2,
	\end{equation*}
where the differential $\mathcal{E}_{\varepsilon,\lambda}'(\omega_{\varepsilon,\lambda})$ is given by
\begin{equation*}
  \mathcal{E}_{\varepsilon,\lambda}'(\omega_{\varepsilon,\lambda})\zeta=\int_S\zeta\left(\mathcal{K}_s\omega-J_\lambda'(\varepsilon^2\omega_{\varepsilon,\lambda})\right)dx.
\end{equation*}
Denote
\begin{equation*}
  \alpha_{\varepsilon,\lambda}=2\mathcal{E}_{\varepsilon,\lambda}'(\omega_{\varepsilon,\lambda})\phi_2,\ \ \mu_{\varepsilon,\lambda}=\mathcal{E}_{\varepsilon,\lambda}'(\omega_{\varepsilon,\lambda})\phi_1,\ \ \psi_{\varepsilon,\lambda}=\mathcal{K}_s\omega_{\varepsilon,\lambda}(x)+\frac{\alpha_{\varepsilon,\lambda}}{2}|x|^2-\mu_{\varepsilon,\lambda}.
\end{equation*}
Then we obtain
\begin{equation*}
  \int_S\zeta\left(\psi_{\varepsilon,\lambda}- J_\lambda'(\varepsilon^2\omega_{\varepsilon,\lambda})\right)dx \le 0.
\end{equation*}
Recalling the condition $\zeta\ge 0$ on $S\backslash D_\delta$, we infer that
\begin{equation*}
  \psi_{\varepsilon,\lambda}=J_\lambda'(\varepsilon^2\omega_{\varepsilon,\lambda})\ \ \text{on}\ D_\delta,\ \ \ \text{and}\ \ \ \psi_{\varepsilon,\lambda}\le J_\lambda'(\varepsilon^2\delta)\ \ \text{on}\ S\backslash D_\delta.
\end{equation*}
Letting $\delta\to 0$, we conclude that
\begin{equation*}
  \psi_{\varepsilon,\lambda}=J_\lambda'(\varepsilon^2\omega_{\varepsilon,\lambda})\ \ \text{on}\ S.
\end{equation*}
Notice that $J_\lambda'(\cdot)$ and $f_\lambda(\cdot)$ are inverse graphs (see \cite{Roc}). It follows that
\begin{equation*}
  \omega_{\varepsilon,\lambda}=\frac{1}{\varepsilon^2} f_\lambda(\psi_{\varepsilon,\lambda})\ \ \text{on}\ S.
\end{equation*}
Finally, we show the uniqueness of $\alpha_{\varepsilon,\lambda}$ and $\mu_{\varepsilon,\lambda}$. Indeed, if there are $\bar{\alpha}_{\varepsilon,\lambda}$ and $\bar{\mu}_{\varepsilon,\lambda}$ such that \eqref{2-5} holds. Then
\begin{equation*}
  \mathcal{K}_s\omega_{\varepsilon,\lambda}(x)+\frac{\alpha_{\varepsilon,\lambda}}{2}|x|^2-\mu_{\varepsilon,\lambda}=\mathcal{K}_s\omega_{\varepsilon,\lambda}(x)+\frac{\bar{\alpha}_{\varepsilon,\lambda}}{2}|x|^2-\bar{\mu}_{\varepsilon,\lambda}
\end{equation*}
for all $x\in \text{supp}(\omega_{\varepsilon,\lambda})$. That is
\begin{equation*}
 \frac{\alpha_{\varepsilon,\lambda}}{2}|x|^2-\mu_{\varepsilon,\lambda}=\frac{\bar{\alpha}_{\varepsilon,\lambda}}{2}|x|^2-\bar{\mu}_{\varepsilon,\lambda}
\end{equation*}
for all $x\in \text{supp}(\omega_{\varepsilon,\lambda})$. Hence we must have $\alpha_{\varepsilon,\lambda}=\bar{\alpha}_{\varepsilon,\lambda}$ and $\mu_{\varepsilon,\lambda}=\bar{\mu}_{\varepsilon,\lambda}$. The proof is thus complete.
\end{proof}

\subsection{Uniformly boundedness in the limit $\lambda\to0$}\label{s2}
Our aim is now to construct solutions of the maximization problems \eqref{maxi} as limits when $\lambda \to 0$ of the function $\omega_{\varepsilon,\lambda}$ of Lemma \ref{lem2-1}. To this end, we first need to establish some uniform estimates with respect to $\lambda$.
\begin{lemma}\label{lem3}
  Let $0< s<1$ and $0<\varepsilon<\min\{\rho_0,\sqrt{7\pi f(1)/72N}\}$. There exists a positive number $\lambda_0$ such that 
  \begin{equation*}
    |\alpha_{\varepsilon,\lambda}|+|\mu_{\varepsilon,\lambda}|+\|\psi_{\varepsilon,\lambda}\|_{L^\infty(\mathbb{R}^2)}+\|\omega_{\varepsilon,\lambda}\|_{L^\infty(S)}\le C
  \end{equation*}
  for any $0<\lambda\le \lambda_0$, where $C$ is a positive number independent of $\lambda$.
\end{lemma}

\begin{proof}
  By \eqref{2-4}, there exists a positive number $C_{s,\varepsilon}$ depending only on $s,\varepsilon$, such that
  \begin{equation*}
    \|\omega_{\varepsilon,\lambda}\|_{L^{1+\frac{1}{s}}(S)}\le C_{s,\varepsilon}.
  \end{equation*}
  It follows that
  \begin{equation*}
    \|\mathcal{K}_s\omega_{\varepsilon,\lambda}\|_{L^\infty(\mathbb{R}^2)}\le C_s \|\omega_{\varepsilon,\lambda}\|_{L^{1+\frac{1}{s}}(S)}\le C_{s,\varepsilon}.
  \end{equation*}
  Hence
  \begin{equation}\label{2-6}
    \frac{\alpha_{\varepsilon,\lambda}}{2}|x|^2-\mu_{\varepsilon,\lambda}-C_{s,\varepsilon}\le \psi_{\varepsilon,\lambda}(x)\le  \frac{\alpha_{\varepsilon,\lambda}}{2}|x|^2-\mu_{\varepsilon,\lambda}+C_{s,\varepsilon},\ \ \ \forall\,x\in \mathbb{R}^2.
  \end{equation}
  Let us first assume that
  \begin{equation*}
    \mu_{\varepsilon,\lambda}\to +\infty,
  \end{equation*}
  as $\lambda\to 0$. In this case, we can suppose that $\mu_{\varepsilon,\lambda}\ge C_{s,\varepsilon}$ for $\lambda$ small enough. It then follows from \eqref{2-6} that $\alpha_{\varepsilon,\lambda} \to +\infty$ as $\lambda\to 0$. Otherwise, $\psi_{\varepsilon,\lambda}$ is non-negative when $\lambda$ is sufficiently small, which is contrary to the constraint $M(\omega_{\varepsilon,\lambda})=1$. Let
    \begin{equation*}
    	S_1:=\left\{x\in S: \frac{1}{2}<|x|<\frac{1}{2}+\frac{1}{6}\right\} \ \ \text{and} \ \ S_2:=\left\{x\in S: \frac{3}{2}-\frac{1}{6}<|x|<\frac{3}{2}\right\}.
    \end{equation*}
  Then
  \begin{equation*}
    \text{meas}\,(S_1)=\frac{7\pi}{72N},\ \ \ \ \text{meas}\,(S_2)=\frac{17\pi}{72N}.
  \end{equation*}
Due to the constraint $L(\omega_{\varepsilon,\lambda})=1$, there exists a point $y^1\in S$ with $|y^1|\le 1$ such that $\psi_{\varepsilon,\lambda}(y^1)\ge 0$. In another direction, since $0<\varepsilon<\sqrt{7\pi f(1)/72N}$, there exists a point $y^2\in S_2$ such that $\psi_{\varepsilon,\lambda}(y^2)\le1$. Otherwise, it will contradict the constraint $M(\omega_{\varepsilon,\lambda})=1$. Hence, by \eqref{2-6}, we have
\begin{equation*}
  \frac{\alpha_{\varepsilon,\lambda}}{2}|y^2|^2-\mu_{\varepsilon,\lambda}-C_{s,\varepsilon}-1\le \psi_{\varepsilon,\lambda}(y^2)-1\le \psi_{\varepsilon,\lambda}(y^1)\le  \frac{\alpha_{\varepsilon,\lambda}}{2}|y^1|^2-\mu_{\varepsilon,\lambda}+C_{s,\varepsilon},
\end{equation*}
  which implies
  \begin{equation*}
    \frac{\alpha_{\varepsilon,\lambda}}{2}\left(|y^2|^2-|y^1|^2 \right)\le 2C_{s,\varepsilon}+1.
  \end{equation*}
  By letting $\lambda\to 0$, we get a contradiction. Hence, $\mu_{\varepsilon,\lambda}$ is bounded from above for $\lambda$ small. One can argue similarly to see that $\mu_{\varepsilon,\lambda}$ is bounded from below for $\lambda$ small. Therefore, we have proved that $\mu_{\varepsilon,\lambda}$ is uniformly bounded for $\lambda$ small. Thanks to the constraint $M(\omega_{\varepsilon,\lambda})=1$, \eqref{2-5} and \eqref{2-6}, the numbers $\alpha_{\varepsilon,\lambda}$ must remain bounded in the limit $\lambda\to 0$. Now, it follows from \eqref{2-6} that there exists a positive number $C$ independent of $\lambda$, such that $|\psi_{\varepsilon,\lambda}(x)|\le C$ for all $x\in \mathbb{R}^2$. Recalling \eqref{2-5}, we conclude that $\|\omega_{\varepsilon,\lambda}\|_{L^\infty(S)}$ remains bounded in the limit $\lambda\to 0$. The proof is thus complete.
\end{proof}

\subsection{Convergence in the limit $\lambda \to 0$}\label{s3} With Lemma \ref{lem3} in hand, we can put forward a compactness argument in order to take the limit $\lambda \to 0$.
\begin{lemma}\label{lem4}
  Let $0< s<1$ and $0<\varepsilon<\min\{\rho_0,\sqrt{7\pi f(1)/72N}\}$. Then there exists an $\omega_{\varepsilon}\in \mathcal{A}$, which is Steiner symmetric with respect to $\theta$, such that
\begin{equation*}
  \mathcal{E}_{\varepsilon}(\omega_{\varepsilon})=\sup_{\omega\in \mathcal{A}}\mathcal{E}_{\varepsilon}(\omega)<+\infty.
\end{equation*}
Moreover, there exist two numbers $\alpha_{\varepsilon}$ and $\mu_{\varepsilon}$ such that
\begin{equation}\label{2-7}
		\omega_{\varepsilon}(x)=\frac{1}{\varepsilon^2} f(\psi_{\varepsilon}(x)),\ \ \text{a.e.}\ x\in S,
\end{equation}
with
\begin{equation*}
  \psi_{\varepsilon}(x)=\mathcal{K}_s\omega_{\varepsilon}(x)+\frac{\alpha_{\varepsilon}}{2}|x|^2-\mu_{\varepsilon}.
\end{equation*}
\end{lemma}

\begin{proof}
By Lemma \ref{lem3}, up to an omitted subsequence, there exist two numbers $\alpha_{\varepsilon}$ and $\mu_{\varepsilon}$ such that
\begin{equation}\label{2-8}
  \alpha_{\varepsilon,\lambda}\to \alpha_{\varepsilon}\ \ \ \text{and}\ \ \ \mu_{\varepsilon,\lambda}\to\mu_{\varepsilon},
\end{equation}
when $\lambda \to 0$. Moreover, there exists a function $\omega_{\varepsilon}\in \mathcal{A}$ such that
\begin{equation}\label{2-9}
  \omega_{\varepsilon,\lambda}\to \omega_{\varepsilon}\ \ \ \text{weakly-star in}\ L^\infty(S)
\end{equation}
when $\lambda \to 0$. It is easy to see that $\omega_{\varepsilon}$ is also Steiner symmetric with respect to $\theta$. By the standard regularity theory, for every $1<p<\infty$, there exists a positive number $C_{s,p}$, depending only on $s$ and $p$, such that
  \begin{equation*}
    \|\mathcal{K}_s\omega_{\varepsilon,\lambda}\|_{\dot{W}^{2s,p}(\mathbb{R}^2)}\le C_{s,p}\|\omega_{\varepsilon,\lambda}\|_{L^p(S)},
  \end{equation*}
Combining this with \eqref{2-6}, we conclude that there exist positive numbers $C(R)$, not depending on $\lambda$, such that
\begin{equation*}
  \|\psi_{\varepsilon,\lambda}\|_{{W}^{2s,p}(B_R(0))}\le C(R),
\end{equation*}
for any $1<p<\infty$ and any positive number $R$. Thus, there exists a continuous function $\psi_\varepsilon:\mathbb{R}^2\to\mathbb{R}$ such that, up to a further subsequence,
\begin{equation*}
  \psi_{\varepsilon,\lambda}\to \psi_\varepsilon\ \ \ \text{in}\ L^\infty(B_R(0)),
\end{equation*}
as $\lambda \to 0$, for any positive number $R$. Clearly, we have
\begin{equation*}
  \psi_\varepsilon(x)=\mathcal{K}_s\omega_{\varepsilon}(x)+\frac{\alpha_{\varepsilon}}{2}|x|^2-\mu_{\varepsilon}.
\end{equation*}
Moreover, by virtue of the $\theta$-symmetrization of $\omega_\varepsilon$, we see that $\psi_\varepsilon$ is strictly symmetric decreasing with respect to $\theta$ in $S$. It follows that every level set of $\psi_\varepsilon$ in $S$ has measure zero. In view of \eqref{2-5}, it holds
\begin{equation}\label{2-10}
  \omega_{\varepsilon,\lambda}(x)=\frac{1}{\varepsilon^2} f_\lambda(\psi_{\varepsilon,\lambda}(x))=\frac{1}{\varepsilon^2} f(\psi_{\varepsilon,\lambda}(x))+\frac{\lambda}{\varepsilon^2}(\psi_{\varepsilon,\lambda})_+^s,\ \ \ \ \text{a.e.}\ x\in S.
\end{equation}
Since $f$ is a monotonic function, it has at most countable discontinuities. Letting $\lambda\to 0$ in \eqref{2-10}, we clearly have
\begin{equation*}
  \omega_{\varepsilon}(x)=\frac{1}{\varepsilon^2} f(\psi_{\varepsilon}(x)),\ \ \text{a.e.}\ x\in S.
\end{equation*}
It remains to be proved that $\omega_\varepsilon$ is a maximizer of $\mathcal{E}_{\varepsilon}$ relative to $\mathcal{A}$. For any $\omega\in \mathcal{A}$ such that $\mathcal{E}_{\varepsilon}(\omega)\not=-\infty$, we have
\begin{equation*}
  \mathcal{E}_{\varepsilon}(\omega)\le  \mathcal{E}_{\varepsilon,\lambda}(\omega)\le \mathcal{E}_{\varepsilon,\lambda}(\omega_{\varepsilon,\lambda}),
\end{equation*}
since $\mathcal{J}_{\varepsilon, \lambda}(\omega)\ge \mathcal{J}_{\varepsilon}(\omega)$. So it suffices to show that
\begin{equation*}
  \lim_{\lambda\to 0}\mathcal{E}_{\varepsilon,\lambda}(\omega_{\varepsilon,\lambda})=\mathcal{E}_{\varepsilon}(\omega_\varepsilon).
\end{equation*}
Since $K_s(\cdot,\cdot)\in L^{1+s}(S\times S)$, we have
\begin{equation}\label{2-11}
  \lim_{\lambda\to+\infty}E_s(\omega_{\varepsilon,\lambda})=E_s(\omega_{\varepsilon}).
\end{equation}
On the other hand, by convexity we have (see \cite{Roc})
\begin{equation*}
  J_\lambda(\varepsilon^2\omega_{\varepsilon,\lambda}(x))=\varepsilon^2\omega_{\varepsilon,\lambda}(x)\psi_{\varepsilon,\lambda}(x)-F_\lambda(\psi_{\varepsilon,\lambda}(x)), \ \ \ x\in S.
\end{equation*}
Hence
\begin{equation*}
  \mathcal{J}_{\varepsilon, \lambda}(\omega_{\varepsilon,\lambda})=\varepsilon^2\int_S\omega_{\varepsilon,\lambda}(x)\psi_{\varepsilon,\lambda}(x)dx-\int_SF_\lambda(\psi_{\varepsilon,\lambda}(x))dx,
\end{equation*}
from which it follows that
\begin{equation}\label{2-12}
  \lim_{\lambda\to 0}\mathcal{J}_{\varepsilon, \lambda}(\omega_{\varepsilon,\lambda})=\varepsilon^2\int_S\omega_{\varepsilon}(x)\psi_{\varepsilon}(x)dx-\int_SF(\psi_{\varepsilon}(x))dx=\mathcal{J}_{\varepsilon}(\omega_{\varepsilon}).
\end{equation}
Combining \eqref{2-11} and \eqref{2-12}, we get
\begin{equation*}
  \lim_{\lambda\to 0}\mathcal{E}_{\varepsilon,\lambda}(\omega_{\varepsilon,\lambda})=\mathcal{E}_{\varepsilon}(\omega_\varepsilon).
\end{equation*}
The proof is thus complete.
\end{proof}

\subsection{Description of $\text{supp}(\omega_\varepsilon)$ when $\varepsilon$ is sufficiently small}\label{s4}
 We note that $\omega_\varepsilon$ obtained in Lemma \ref{lem4} is not yet sufficient to provide a dynamically possible steady vortex flow. To get a desired solution, we need to prove that the support of $\omega_\varepsilon$ is away from the boundary of $S$ (see Lemma \ref{lem15} below). We will show that this is true  when $\varepsilon$ is sufficiently small. It is based on the observation that in order to maximize energy, the support of a maximizer can not be too scattered. We will reach this goal by several steps. For convenience, we will use $C$ below to denote various positive constants not depending on $\varepsilon$ that may change from line to line.
 We begin by giving a lower bound of $\mathcal{E}_\varepsilon(\omega_\varepsilon)$. Let
\begin{equation*}
  A_s=c_s\pi^{-1-s}\int_{B_1(0)}\int_{B_1(0)}\frac{1 }{|x-x'|^{2-2s}}dxdx'.
\end{equation*}

\begin{lemma}\label{lem5}
For every $\delta\in (0,1)$, there exists a positive number $C_\delta$, depending only on $\delta$, such that
\begin{equation*}
  \mathcal{E}_\varepsilon(\omega_\varepsilon)\ge \frac{\delta^{1-s}A_s}{2\varepsilon^{2-2s}}-C_\delta.
\end{equation*}
\end{lemma}

\begin{proof}
  The key idea is to choose a suitable test function. Let $\delta\in (0,1)$. Set
\begin{equation*}
  \varpi^\delta_\varepsilon=\frac{\delta}{\varepsilon^2}\textbf{1}_{B_{\varepsilon/\sqrt{\pi\delta}}(a^\delta_\varepsilon, 0)},\ \ \text{with}\ \ a^\delta_\varepsilon=\sqrt{1-{\varepsilon^2}/{2\delta\pi}}.
\end{equation*}
We check that $\varpi^\delta_\varepsilon\in \mathcal{A}$. A direct calculation then yields
\begin{equation}\label{2-13}
  E_s(\varpi^\delta_\varepsilon)\ge \frac{\delta^{1-s}A_s}{2\varepsilon^{2-2s}}-C
\end{equation}
for some positive number $C$ independent of $\varepsilon$ and $\delta$. Notice that the effect domain of $J(\cdot)$ is contained in $(-\infty, \sup_\mathbb{R} f]$. Thus there exists a positive number $C_\delta$, depending only on $\delta$, such that
\begin{equation}\label{2-14}
  \mathcal{J}_\varepsilon(\varpi^\delta_\varepsilon)\le C_\delta.
\end{equation}
Combining \eqref{2-13} and \eqref{2-14}, we obtain
\begin{equation*}
   \mathcal{E}_\varepsilon(\omega_\varepsilon)\ge \mathcal{E}_\varepsilon(\varpi^\delta_\varepsilon)\ge \frac{\delta^{1-s}}{2}\frac{A_s}{\varepsilon^{2-2s}}-C_\delta,
\end{equation*}
which completes the proof.
\end{proof}

Let
\begin{equation}\label{zzz0}
  \zeta_\varepsilon(x)=\varepsilon^2\omega_\varepsilon(\varepsilon x),\ \ \ x\in \mathbb{R}^2.
\end{equation}
Then
\begin{equation}\label{zzz}
  0\le \zeta_\varepsilon\le 1, \ \ \ \text{supp}(\zeta_\varepsilon)\subset B_{2/\varepsilon}(0)\ \ \   \text{and}\ \ \ \int_{\mathbb{R}^2}\zeta_\varepsilon dx=1.
\end{equation}
Observe that
\begin{equation*}
  E_s(\omega_\varepsilon)=\frac{c_s}{2\varepsilon^{2-2s}}\int_{\mathbb{R}^2}\int_{\mathbb{R}^2}\frac{\zeta_\varepsilon(x)\zeta_\varepsilon(x')}{|x-x|^{2-2s}}dxdx'+O(1).
\end{equation*}
Let
\begin{equation}\label{zzz2}
  \mathcal{I}_s(\zeta)=c_s\int_{\mathbb{R}^2}\int_{\mathbb{R}^2}\frac{\zeta(x)\zeta(x')}{|x-x'|^{2-2s}}dxdx'.
\end{equation}

From Lemma \ref{lem5}, we can obtain the following result.
\begin{lemma}\label{lem6}
For every $\delta\in (0,1)$, there exists a positive number $C_\delta$, depending only on $\delta$, such that
\begin{equation}\label{2-15}
  \mathcal{I}_s(\zeta_\varepsilon)\ge \delta^{1-s}A_s-C_\delta\varepsilon^{2-2s}.
\end{equation}
\end{lemma}

The following lemma shows that the support of maximizer $\omega_\varepsilon$ can not be too scattered.
\begin{lemma}\label{lem7}
  For arbitrary $\eta\in (0,1)$, there exists a positive number $R$, such that
  \begin{equation*}
     \sup_{y\in \mathbb{R}^2}\int_{B_R(y)}\zeta_\varepsilon dx> 1-\eta,\ \ \ \forall\,\varepsilon>0.
  \end{equation*}
\end{lemma}

To prove Lemma \ref{lem7}, we need two auxiliary lemmas. The first is the concentration compactness lemma, which is due to
Lions \cite{Lions}.

\begin{lemma}\label{lem8}
	Let $\{\xi_n\}_{n=1}^\infty$ be a sequence of nonnegative functions in $L^1(\Pi)$ satisfying
	$$\limsup_{n\rightarrow \infty} \int_{\mathbb{R}^2} \xi_n dx\rightarrow \beta,$$ for some $0< \beta<\infty$.
	Then, after passing to a subsequence, one of the following holds:\\
\begin{itemize}
  \item [(i)](Compactness) There exists a sequence $\{y_n\}_{n=1}^\infty$ in $\mathbb{R}^2$ such that for arbitrary $\epsilon>0$, there exists $R>0$ satisfying
	\begin{equation*}
		\int_{B_R(y_n)}\xi_n dx\geq \beta-\epsilon, \quad \forall\, n\geq 1.
	\end{equation*}\\
  \item [(ii)](Vanishing) For each $R>0$,
	\begin{equation*}
		\lim_{n\rightarrow \infty}\sup_{y\in \mathbb{R}^2}  \int_{B_R(y)} \xi_n dx =0.
	\end{equation*} \\
  \item [(iii)](Dichotomy) There exists a number $0<\beta_1<\beta$ such that for any $\epsilon>0$, there exist $N=N(\epsilon)\geq 1$ and $0\leq \xi_{i,n}\leq \xi_n, \,i=1,2$ satisfying
	\begin{equation*}
		\begin{cases}
			\|\xi_n-\xi_{1,n}-\xi_{2,n}\|_1+|\beta-\int_{\mathbb{R}^2} \xi_{1,n} dx|+|\beta-\beta_1-\int_{\mathbb{R}^2} \xi_{2,n} dx|<\epsilon,\quad \text{for}\ n\geq N,\\
			d_n:=\text{dist}(\text{supp}(\xi_{1,n}), \text{supp}(\xi_{2,n}))\rightarrow \infty, \quad \text{as}\ n\rightarrow \infty.
		\end{cases}	
	\end{equation*}
\end{itemize}
\end{lemma}

In addition, we need the following result.
\begin{lemma}[Bathtub principle for Riesz integrals]\label{bath}
Let $s\in (0,1)$. Let $\beta_1$ and $\beta_2$ be two positive numbers. Given a positive number $\eta$, set
\begin{equation*}
  \mathcal{B}_\eta:=\big{\{}\xi \in L^\infty(\mathbb{R}^2,[0,1]): \int_{\mathbb{R}^2}\xi dx\le \eta \big{\}}.
  \end{equation*}
  Then the maximization problem
  \begin{equation*}
    \mathcal{M}_{\beta_1,\beta_2}:=\sup_{(\xi_1,\xi_2)\in \mathcal{B}_{\eta_1}\times \mathcal{B}_{\eta_2}}\int_{\mathbb{R}^2}\int_{\mathbb{R}^2} \frac{1}{|x-y|^{2-2s}}\xi_1(x)\xi_2(y) dxdy
  \end{equation*}
is solved by the functions
\begin{equation*}
  \xi_i(x)=\textbf{1}_{B_{R_i}(0)},\ \ \ R_i=\sqrt{\beta_i/\pi},\ \ i=1,2.
\end{equation*}
Moreover, this solution is unique up to a translation.
\end{lemma}
\begin{proof}
  This is a simple consequence of the bath principle and the Riesz rearrangement inequality; see, e.g., Theorem 1.14 and Theorem 3.9 in \cite{Lieb}.
\end{proof}

Now, we can give the proof of Lemma \ref{lem7}.
\begin{proof}[Proof of Lemma \ref{lem7}]
  We argue by contradiction. If the statement was false, then there exists a number $0<\eta_0<1$ such that for any integer $n\ge1$, there exists an $\varepsilon_n>0$ satisfying $\varepsilon_n\to 0^+$ as $n\to \infty$, such that
  \begin{equation*}
    \sup_{y\in \mathbb{R}^2}\int_{B_n(y)}\zeta_{\varepsilon_n} dx\le1-\eta_0, \ \ \ \forall\,n\ge1.
  \end{equation*}
  By virtue of Lemma \ref{lem8} with $\xi_n=\zeta_{\varepsilon_n}$, we conclude that for a certain subsequence, still denoted by $\{\zeta_{\varepsilon_n}\}_{n=1}^\infty$, one of the three cases in Lemma \ref{lem8} should occur. If $\{\zeta_{\varepsilon_n}\}_{n=1}^\infty$ had the Compactness Property, then it would contradict the hypothesis and the proof is thus finished. So it suffices to exclude the Vanishing Property and the Dichotomy Property.

  \emph{Step 1. Vanishing excluded:}
Suppose for each fixed $R>0$,
\begin{equation}\label{2-16}
	\lim_{n\rightarrow \infty}\sup_{y\in \mathbb{R}^2}  \int_{B_R(y)} \zeta_{\varepsilon_n} dx =0.
\end{equation}
  We will show $\lim_{n\rightarrow \infty} \mathcal{I}_s(\zeta_{\varepsilon_n})=0$, which contradicts \eqref{2-15}. Indeed, for any $x\in \mathbb{R}^2$ and $R>0$, we have
  \begin{equation*}
  \begin{split}
     \int_{\mathbb{R}^2}\frac{\zeta_{\varepsilon_n}(x')}{|x-x'|^{2-2s}}dx' & \le \int_{|x'-x|<R}  \frac{\zeta_{\varepsilon_n}(x')}{|x-x'|^{2-2s}}dx'  + \int_{|x'-x|\ge R} \frac{\zeta_{\varepsilon_n}(x')}{|x-x'|^{2-2s}}dx' \\
       & \le \left(\int_{|x'-x|<R}\frac{1}{|x-x'|^{2-2s^2}}dx'\right)^{\frac{1}{1+s}} \|\zeta_{\varepsilon_n}\|_{L^{1+\frac{1}{s}}(B_R(x))}+\frac{1}{R^{2-2s}}\\
       & \le C_sR^{2s^2}\left(\sup_{y\in \mathbb{R}^2}  \int_{B_R(y)} \zeta_{\varepsilon_n} dx\right)^{\frac{s}{1+s}}+\frac{1}{R^{2-2s}},\\
  \end{split}
  \end{equation*}
  where $C_s$ is a positive number depending only on $s$. Hence
  \begin{equation*}
  \begin{split}
     \mathcal{I}_s(\zeta_{\varepsilon_n}) & =c_s\int_{\mathbb{R}^2}\int_{\mathbb{R}^2}\frac{\zeta_{\varepsilon_n}(x)\zeta_{\varepsilon_n}(x')}{|x-x'|^{2-2s}}dxdx' \\
       &\le C_sR^{2s^2}\left(\sup_{y\in \mathbb{R}^2}  \int_{B_R(y)} \zeta_{\varepsilon_n} dx\right)^{\frac{s}{1+s}}+\frac{C_s}{R^{2-2s}}.
  \end{split}
  \end{equation*}
  In view of \eqref{2-16}, we infer from the above inequality by first letting $n\rightarrow \infty$, then $R\rightarrow \infty$ that
  \begin{equation*}
    \lim_{n\rightarrow \infty} \mathcal{I}_s(\zeta_{\varepsilon_n})=0.
  \end{equation*}

  \emph{Step 2. Dichotomy excluded:}
  Suppose there exists a number $0<\beta<1$ such that for any $\epsilon>0$, there exist $N(\epsilon)\geq 1$ and $0\leq \zeta_{i,\varepsilon_n}\leq \zeta_{\varepsilon_n}, \,i=1,2,3$ satisfying
\begin{equation*}
	\begin{cases}
		\zeta_{\varepsilon_n}=\zeta_{1,\varepsilon_n}+\zeta_{2,\varepsilon_n}+\zeta_{3,\varepsilon_n},\\
		|\beta-\beta_{1,n}|+|1-\beta-\beta_{2,n}|+|\beta_{3,n}|<\epsilon,\quad \text{for}\ n\geq N(\epsilon),\\
		d_n:=\text{dist}(\text{supp}(\zeta_{1,\varepsilon_n}), \text{supp}(\zeta_{2,\varepsilon_n}))\rightarrow \infty, \quad \text{as}\ n\rightarrow \infty,
	\end{cases}	
\end{equation*}
where $\beta_{i,n}=\|\zeta_{1,\varepsilon_n}\|_{L^1(\mathbb{R}^2)}$, $i=1,2,3$. Using the diagonal argument, we obtain that there exists a subsequence, still denoted by $\{\zeta_{\varepsilon_n}\}_{n=1}^\infty$, such that
\begin{equation}\label{2-17}
	\begin{cases}
		\zeta_{\varepsilon_n}=\zeta_{1,\varepsilon_n}+\zeta_{2,\varepsilon_n}+\zeta_{3,\varepsilon_n},\\
		|\beta-\beta_{1,n}|+|1-\beta-\beta_{2,n}|+|\beta_{3,n}|\to 0,\quad \text{as}\ n\rightarrow \infty,\\
		d_n:=\text{dist}(\text{supp}(\zeta_{1,\varepsilon_n}), \text{supp}(\zeta_{2,\varepsilon_n}))\rightarrow \infty, \quad \text{as}\ n\rightarrow \infty.
	\end{cases}	
\end{equation}
In view of \eqref{zzz}, it must hold $\varepsilon_n\to 0^+$ when $n\to \infty$.
By the symmetry of $\mathcal{I}_s$, we have
\begin{equation*}
  \begin{split}
     \mathcal{I}_s(\zeta_{\varepsilon_n}) &= \mathcal{I}_s(\zeta_{1,\varepsilon_n}+\zeta_{2,\varepsilon_n}+\zeta_{3,\varepsilon_n})\\
	&=c_s\int_{\mathbb{R}^2}\int_{\mathbb{R}^2} \frac{\zeta_{1,\varepsilon_n}(x)\zeta_{1,\varepsilon_n}(x')}{|x-x'|^{2-2s}}dxdx'+c_s\int_{\mathbb{R}^2}\int_{\mathbb{R}^2} \frac{\zeta_{2,\varepsilon_n}(x)\zeta_{2,\varepsilon_n}(x')}{|x-x'|^{2-2s}}dxdx'\\
 & \ \ \ +2c_s\int_{\mathbb{R}^2}\int_{\mathbb{R}^2} \frac{\zeta_{1,\varepsilon_n}(x)\zeta_{2,\varepsilon_n}(x')}{|x-x'|^{2-2s}}dxdx'+c_s\int_{\mathbb{R}^2}\int_{\mathbb{R}^2} \frac{\left(2\zeta_{\varepsilon_n}(x)-\zeta_{3,\varepsilon_n}(x)\right)\zeta_{3,\varepsilon_n}(x')}{|x-x'|^{2-2s}}dxdx'\\
 &=:I_1+I_2+I_3+I_4.
  \end{split}
\end{equation*}
In view of \eqref{zzz}, using the bathtub principle for Riesz integrals, we have
\begin{equation*}
  I_i\le \beta^{1+s}_{i,n}A_s, \ \ \ \ \ i=1,2.
\end{equation*}
For $I_3$, we clearly have
\begin{equation*}
  I_3\le 2c_s d_n^{2s-2}=o_n(1),
\end{equation*}
as $n\to \infty$. Finally, we have
\begin{equation*}
  \begin{split}
     I_4 & \le 2c_s\int_{\mathbb{R}^2}\int_{\mathbb{R}^2} \frac{\zeta_{\varepsilon_n}(x)\zeta_{3,\varepsilon_n}(x')}{|x-x'|^{2-2s}}dxdx' \\
       & = 2c_s\int_{\mathbb{R}^2}\zeta_{\varepsilon_n}(x) \left(\int_{|x-x'|\ge1}+\int_{|x-x'|<1} \frac{\zeta_{3,\varepsilon_n}(x')}{|x-x'|^{2-2s}}dx'\right)dx\\
       & \le 2c_s\int_{\mathbb{R}^2}\zeta_{\varepsilon_n}(x) \left(\|\zeta_{3,\varepsilon_n}\|_{L^1(\mathbb{R}^2)}+C\|\zeta_{3,\varepsilon_n}\|_{L^{1+\frac{1}{s}}(\mathbb{R}^2)}     \right)dx \\
       & \le C_s\left(\|\zeta_{3,\varepsilon_n}\|_{L^1(\mathbb{R}^2)}+\|\zeta_{3,\varepsilon_n}\|^\frac{s}{1+s}_{L^1(\mathbb{R}^2)}\right)\\
       &=o_n(1),
  \end{split}
\end{equation*}
as $n\to \infty$. It follows from the above calculations that
\begin{equation*}
  \mathcal{I}_s(\zeta_{\varepsilon_n})\le \left(\beta^{1+s}+(1-\beta)^{1+s}\right)A_s+o_n(1),
\end{equation*}
as $n\to \infty$. Since $0<\beta<1$, we can choose a $\delta_0\in (0,1)$ such that
  \begin{equation*}
    \beta^{1+s}+(1-\beta)^{1+s}<\delta^{1-s}_0.
  \end{equation*}
Hence
\begin{equation}\label{2-18}
  \mathcal{I}_s(\zeta_{\varepsilon_n})\le \delta^{1-s}_0A_s+o_n(1),
\end{equation}
as $n\to \infty$. On the other hand, by Lemma \ref{lem6}, we have
\begin{equation}\label{2-19}
  \mathcal{I}_s(\zeta_{\varepsilon_n})\ge \left(\frac{\delta_0+1}{2}\right)^{1-s}A_s+o_n(1),
\end{equation}
as $n\to \infty$. By letting $n\to \infty$, we infer from \eqref{2-18} and \eqref{2-19} that
\begin{equation*}
  \frac{\delta_0+1}{2}\le \delta_0.
\end{equation*}
This contradicts the fact that $\delta_0\in (0,1)$. The proof is thus complete.
\end{proof}

\begin{lemma}\label{lem9}
  There exists a family of points $\{y_\varepsilon\}_{\varepsilon>0}$ on the $x_1$-axis, such that for arbitrary $\eta\in(0,1)$, there exists a positive number $R_\eta$ satisfying
  \begin{equation*}
     \int_{B_{R_\eta}(y_\varepsilon)}\zeta_\varepsilon dx> 1-\eta,\ \ \ \forall\,\varepsilon>0.
  \end{equation*}
\end{lemma}

\begin{proof}
  From Lemma \ref{lem7}, we know that there exists a family of points $\{y^0_\varepsilon\}_{\varepsilon>0}$ in $\mathbb{R}^2$ and $R_0$
  \begin{equation*}
     \int_{B_{R_0}(y^0_\varepsilon)}\zeta_\varepsilon dx> \frac{5}{6},\ \ \ \forall\,\varepsilon>0.
  \end{equation*}
  Note that $\zeta_\varepsilon$ is Steiner symmetric with respect to $\theta$. So we may assume that $y^0_\varepsilon\in \mathbb{R}\times\{0\}$. For arbitrary $\eta\in(0,1/2)$, we will show that there exists some $R_\eta>R_0$, such that
  \begin{equation*}
     \int_{B_{R_\eta}(y^0_\varepsilon)}\zeta_\varepsilon dx>1-\eta,\ \ \ \forall\,\varepsilon>0.
  \end{equation*}
  In fact, by Lemma \ref{lem7}, there exists a $\tilde{R}_\eta>R_0$ and a family of points $\{y^\eta_\varepsilon\}_{\varepsilon>0}$ in $\mathbb{R}^2$, such that
  \begin{equation*}
     \int_{B_{\tilde{R_\eta}}(y^\eta_\varepsilon)}\zeta_\varepsilon dx>1-\eta,\ \ \ \forall\,\varepsilon>0.
  \end{equation*}
  We now prove that $|y^\eta_\varepsilon-y^0_\varepsilon|\le 2\tilde{R}_\eta$ for all $\varepsilon>0$. Suppose not, then $B_{\tilde{R}_\eta}(y^0_\varepsilon)\cap B_{\tilde{R}_\eta}(y^\eta_\varepsilon)=\varnothing$, and we have
  \begin{equation*}
    1=\int_{\mathbb{R}^2}\zeta_\varepsilon dx\ge \int_{B_{\tilde{R}_\eta}(y^0_\varepsilon)}\zeta_\varepsilon dx+\int_{B_{\tilde{R}_\eta}(y^\eta_\varepsilon)}\zeta_\varepsilon dx\ge \frac{5}{6}+1-\eta>1.
  \end{equation*}
  This is a contradiction. Let $R_\eta=3\tilde{R}_\eta$. Then we have
  \begin{equation*}
     \int_{B_{R_\eta}(y^0_\varepsilon)}\zeta_\varepsilon dx\ge \int_{B_{\tilde{R}_\eta}(y^\eta_\varepsilon)}\zeta_\varepsilon dx>1-\eta,\ \ \ \forall\,\varepsilon>0.
  \end{equation*}
\end{proof}

Coming back to the original functions $\{\omega_\varepsilon\}_{\varepsilon>0}$, we immediately obtain the following lemma.
\begin{lemma}\label{lem10}
   There exists a family of points $\{z_\varepsilon\}_{\varepsilon>0}$ on the $x_1$-axis, such that for arbitrary $\eta\in(0,1)$, there exists a positive number $R_\eta$ satisfying
  \begin{equation*}
     \int_{B_{R_\eta\varepsilon}(z_\varepsilon)}\omega_\varepsilon dx> 1-\eta,\ \ \ \forall\,\varepsilon>0.
  \end{equation*}
  Moreover, $z_\varepsilon\to(1,0)$ when $\varepsilon \to 0^+$.
\end{lemma}
\begin{proof}
  Recalling \eqref{zzz0}, the existence of $\{z_\varepsilon\}_{\varepsilon>0}$ clearly follows from Lemma \ref{lem9}. By virtue of the constraint $L(\omega_\varepsilon)=1$, it is easy to see that $z_\varepsilon\to(1,0)$ when $\varepsilon \to 0^+$. The proof is thus complete.
\end{proof}

Set
	\begin{equation*}
		\rho:=\min\left\{\frac{1}{6}, \ \frac{1}{2}\sin\left(\frac{\pi}{2N}\right)\right\}.
	\end{equation*}
The first localization result is as follows.
\begin{lemma}\label{lem11}
For all sufficiently small $\varepsilon>0$, it holds
  \begin{equation}\label{2-20}
    \text{supp}(\omega_\varepsilon)\subseteq B_\rho\left((1,0)\right)\cup D,
  \end{equation}
  where
  \begin{equation*}
    D:=\Big{\{}x\in S: {1}/{2}\le |x|\le {1}/{2}+{1}/{6}\ \,\text{or}\, \ {3}/{2}-{1}/{6}\le |x|\le {3}/{2}\Big{\}}.
  \end{equation*}
  Moreover,
  \begin{equation}\label{2-21}
    \lim_{\varepsilon\to 0^+}\int_{ S\backslash B_{\rho/2}\left((1,0)\right)} \omega_\varepsilon dx=0.
  \end{equation}
\end{lemma}

\begin{proof}
  From Lemma \ref{lem10}, we see that for every $\sigma\in (0,1)$, there exists a positive number $R_\sigma>[f(1)\pi]^{-1/2}$ (independent of $\varepsilon$), such that
  \begin{equation*}
     \int_{B_{R_\sigma\varepsilon}(z_\varepsilon)}\omega_\varepsilon dx> 1-\sigma,\ \ \ \forall\,\varepsilon>0.
  \end{equation*}
  If $\sigma$ is fixed, then
\begin{equation*}
    	B_{R_\sigma\varepsilon}(z_\varepsilon)\subset \left\{x\in S: 1-\frac{\rho}{2}<|x|<1+\frac{\rho}{2}\right\}
    \end{equation*}
for $\varepsilon$ small. It follows that
\begin{equation*}
   \lim_{\varepsilon\to 0^+}\int_{ S\backslash B_{\rho/2}\left((1,0)\right)} \omega_\varepsilon dx=0.
\end{equation*}
Thus \eqref{2-20} is proved. By the definition of $K_s$, we have
\begin{equation}\label{2-25}
  \mathcal{K}_s\omega_\varepsilon(x)=\int_{S}K_s(x,x')\omega_\varepsilon(x')dx'\ge\frac{c_s}{(2R_\sigma)^{2-2s}}\frac{1-\sigma}{\varepsilon^{2-2s}} \ \ \ \text{whenever}\ \ x\in B_{R_\sigma\varepsilon}(z_\varepsilon).
\end{equation}
On the other hand, by a rearrangement argument, for every $\eta>0$ and all sufficiently small $\varepsilon$ we have
\begin{equation}\label{2-26}
    	\mathcal{K}_s\omega_\varepsilon(x)\le\frac{\eta}{\varepsilon^{2-2s}}+C, \ \ \ \text{whenever}\ \ x\in S\backslash B_\rho\left((1,0)\right).
    \end{equation}
    We now prove that $\psi_\varepsilon<0$ in $S'=S\backslash\left(B_\rho\left((1,0)\right)\cup D\right)$ if $\varepsilon$ is sufficiently small. We argue by contradiction. If the statement was false, then there exists a point $\bar{x}\in S'$ satisfying $\psi_\varepsilon(x)\ge 0$. Then for any $y\in S$ such that $\psi_\varepsilon(y)\le 1$ we have clearly
    \begin{equation}\label{2-28}
      1\ge \psi_\varepsilon(y)-\psi_\varepsilon(\bar{x})=\mathcal{K}_s\omega_\varepsilon(y)-\mathcal{K}_s\omega_\varepsilon(\bar{x})+\frac{\alpha_\varepsilon}{2}\left(|y|^2-|\bar{x}|^2\right).
    \end{equation}
    When $\varepsilon$ is small enough, we can find a point $y^1\in S\backslash B_\rho\left((1,0)\right)$ such that $\psi_\varepsilon(y_1)\le 1$, and which satisfies the two conditions
    \begin{equation*}
      \text{sign}(\alpha_\varepsilon)=\text{sign}(|y^1|^2-|\bar{x}|^2)\ \ \text{and}\ \ \left||y^1|^2-|\bar{x}|^2\right|\ge \frac{1}{12}.
    \end{equation*}
    Substituting $y^1$ for $y$ in the inequality \eqref{2-28}, we obtain $|\alpha_\varepsilon|\le 24 \left(\mathcal{K}_s\omega_\varepsilon(x)-\mathcal{K}_s\omega_\varepsilon(y^1)\right)$. Combining this with \eqref{2-26}, we get
    \begin{equation}\label{2-29}
      |\alpha_\varepsilon|\le \frac{48\eta}{\varepsilon^{2-2s}}+C.
    \end{equation}
    On the other hand, we may choose $y^2\in B_{R_\sigma\varepsilon}(z_\varepsilon)$ such that $\psi_\varepsilon(y^2)\le 1$. Otherwise, we have $\omega_\varepsilon(x)\ge f(1)/\varepsilon^2$ for all $x\in B_{R_\sigma\varepsilon}(z_\varepsilon)$, and hence
    \begin{equation*}
      \int_S \omega_\varepsilon dx\ge \frac{f(1)}{\varepsilon^2}\pi R^2_\sigma\varepsilon^2>1.
    \end{equation*}
     This is contrary to the constraint $M(\omega_\varepsilon)=1$. Now inequality \eqref{2-28} combined with \eqref{2-25} and \eqref{2-29} yields
        \begin{equation}\label{2-30}
    	\begin{split}
    	\frac{c_s}{(2R_\sigma)^{2-2s}}\frac{1-\sigma}{\varepsilon^{2-2s}}\le \mathcal{K}_s\omega_\varepsilon(y^2)&\le \mathcal{K}_s\omega_\varepsilon(\bar{x})+\frac{|\alpha_\varepsilon|}{2}\left(|y^2|^2-|\bar{x}|^2\right)\\
    	&\le \frac{97\eta}{\varepsilon^{2-2s}}+C.
    	\end{split}
    \end{equation}
    Now let $\sigma$ and $\eta$ be fixed so that
    \begin{equation*}
      \frac{c_s(1-\sigma)}{(2R_\sigma)^{2-2s}}>97\eta.
    \end{equation*}
    We get a contradiction from \eqref{2-30} when $\varepsilon$ is small enough. In other words, we have established that $\text{supp}(\omega_\varepsilon)\cap S'=\varnothing$,
which completes the proof.
    \end{proof}

\subsection{Estimates for the Lagrange multipliers $\alpha_\varepsilon$ and $\mu_\varepsilon$}\label{s5}
Now we estimate the Lagrange multiplier $\alpha_\varepsilon$ and $\mu_\varepsilon$. First, we have
\begin{lemma}\label{lem12}Let $1/2\le s<1$. As $\varepsilon\to 0^+$, it holds
	\begin{equation}\label{2-31}
		\alpha_\varepsilon\to \sum\limits_{k=1}^{N-1}\frac{c_s(1-s)}{|(1,0)-Q_{\frac{2k\pi}{N}}|^{2-2s}}.
	\end{equation}
\end{lemma}
\begin{proof}
  Let $y\in \mathbb{R}^2$ be arbitrary. Let $\varphi\in C_0^\infty\left(B_{2\rho}\left((1,0)\right)\right)$ so that $\varphi(x)=y\cdot x$ on $B_\rho\left((1,0)\right)$. Let $F_\varepsilon(\tau):=F(\tau)/\varepsilon^2$. Then $F_\varepsilon(\psi_\varepsilon)=0$ on $\partial B_{2\rho}\left((1,0)\right)$. By integration by parts we get
\begin{equation}\label{2-32}
  \int_{\mathbb{R}^2}\omega_\varepsilon \nabla^\perp(\mathcal{K}_s\omega_\varepsilon+\frac{\alpha_\varepsilon}{2} |x|^2)\cdot\nabla \varphi dx=-\int_S F_\varepsilon\left(\psi_\varepsilon)(\partial_{x_2}\partial_{x_1}\varphi-\partial_{x_1}\partial_{x_2}\varphi\right)dx=0.
\end{equation}
On the other hand, by \eqref{2-20} we have
\begin{equation}\label{2-33}
 \int_{\mathbb{R}^2}\omega_\varepsilon\nabla^\perp\left(\frac{\alpha_\varepsilon}{2} |x|^2\right)\cdot\nabla \varphi dx=\alpha_\varepsilon\int_{B_\rho\left((1,0)\right)}\omega_\varepsilon(x) x^\perp\cdot y dx,
\end{equation}
and
\begin{equation}\label{2-34}
\begin{split}
   \int_{\mathbb{R}^2}&\omega_\varepsilon\nabla^\perp\mathcal{K}_s\omega_\varepsilon \cdot\nabla \varphi dx \\
    =&\int_{B_\rho\left((1,0)\right)}\omega_\varepsilon(x)\nabla^\perp_x\left(\int_{B_\rho\left((1,0)\right)}+\int_{S\backslash B_\rho\left((1,0)\right)}G_s(x-x')\omega_\varepsilon(x')dx'\right)\cdot y dx\\
      =& \int_{B_\rho\left((1,0)\right)}\omega_\varepsilon(x)\nabla^\perp_x \left(\int_{B_\rho\left((1,0)\right)}G_s(x-x')\omega_\varepsilon(x')dx'\right)\cdot y dx+o(1)\\
      = &c_s(s-1)\int_{B_\rho\left((1,0)\right)}\int_{B_\rho\left((1,0)\right)}\omega_\varepsilon(x)\omega_\varepsilon(x')\sum_{k=1}^{N-1}\frac{(x-Q_{\frac{2k\pi}{N}}x')^\perp}{|x-Q_{\frac{2k\pi}{N}}x'|^{4-2s}}\cdot(y-Q_{\frac{2k\pi}{N}}y')dx'dx+o(1).
\end{split}
\end{equation}
In view of Lemmas \ref{lem10} and \ref{lem11}, it then follows from \eqref{2-32}, \eqref{2-33} and \eqref{2-34} that
\begin{equation*}
  \alpha_\varepsilon y_2=c_s(s-1)\sum_{k=1}^{N-1}\frac{\left((1,0)-Q_{\frac{2k\pi}{N}}(1,0)\right)^\perp}{|(1,0)-Q_{\frac{2k\pi}{N}}(1,0)|^{4-2s}}\cdot(y-Q_{\frac{2k\pi}{N}}y')+o(1).
\end{equation*}
 Taking $y=(0,-1)=(1,0)^\perp$, the conclusion \eqref{2-31} follows.
\end{proof}

Next, we estimate the other Lagrange multiplier $\mu_\varepsilon$.
\begin{lemma}\label{lem13}Let $1/2\le s<1$. There holds
   \begin{equation*}
      0<\liminf_{\varepsilon\to0^+}\varepsilon^{2-2s}\mu_\varepsilon\le \limsup_{\varepsilon\to0^+}\varepsilon^{2-2s}\mu_\varepsilon<+\infty.
    \end{equation*}
\end{lemma}

\begin{proof}
  By a rearrangement argument, we first have
  \begin{equation*}
    \mathcal{K}_s\omega_\varepsilon(x)=\int_{S}K_s(x,x')\omega_\varepsilon(x')dx'\le \frac{c_s\pi^{1-s}}{s\varepsilon^{2-2s}}+C
  \end{equation*}
  for any $x\in S$. Notice that
  \begin{equation*}
    \psi_\varepsilon(x)=\mathcal{K}_s\omega_{\varepsilon}(x)+\frac{\alpha_{\varepsilon}}{2}|x|^2-\mu_{\varepsilon}\le \frac{c_s\pi^{1-s}}{s\varepsilon^{2-2s}}-\mu_\varepsilon+C.
  \end{equation*}
  Thus we must have
  \begin{equation*}
    \limsup_{\varepsilon\to0^+}\varepsilon^{2-2s}\mu_\varepsilon\le \frac{c_s\pi^{1-s}}{s}.
  \end{equation*}
  Otherwise, $\psi_\varepsilon<0$ on $S$ and hence $\omega_\varepsilon\equiv 0$, which is absurd.
  On the other hand, by Lemma \ref{lem10}, there exists a positive number $R_0>[f(1)\pi]^{-1/2}$ independent of $\varepsilon$, such that
  \begin{equation*}
    \int_{B_{R_0\varepsilon}(z_\varepsilon)}\omega_\varepsilon dx \ge {1}/{2}.
  \end{equation*}
  Thanks to \eqref{2-7}, there exists a point $x_\varepsilon \in B_{R_0\varepsilon}(z_\varepsilon)$ satisfying $\psi(x_\varepsilon)\le 1$. In view of \eqref{2-25}, we obtain
  \begin{equation*}
    1\ge \psi_\varepsilon(x_\varepsilon)\ge \frac{C_s}{\varepsilon^{2-2s}}-\mu_\varepsilon-C,
  \end{equation*}
  for some positive number $C_s$ depending only on $s$. It follows that
  \begin{equation*}
    \liminf_{\varepsilon\to0^+}\varepsilon^{2-2s}\mu_\varepsilon\ge C_s>0,
  \end{equation*}
  which completes the proof.
\end{proof}

\subsection{Sharp estimate and conclusive localization of the vortex support}\label{s6}
\begin{lemma}\label{lem14}Let $1/2\le s<1$.
 There exists a positive number $\Lambda_0>0$, independent of $\varepsilon$, such that
 \begin{equation*}
   \text{supp}(\omega_\varepsilon)\subseteq B_{\Lambda_0 \varepsilon}\left((1,0)\right).
 \end{equation*}
\end{lemma}

\begin{proof}
  Let $\Lambda>1$ be a number to be determined. By Lemma \ref{lem10}, for every $\sigma\in (0,1)$, there exists a positive number $R_\sigma>0$ (independent of $\varepsilon$), such that
  \begin{equation*}
     \int_{B_{R_\sigma\varepsilon}(z_\varepsilon)}\omega_\varepsilon dx> 1-\sigma,\ \ \ \forall\,\varepsilon>0.
  \end{equation*}
  If $y\in S\backslash B_{2\Lambda R_\sigma\varepsilon}(z_\varepsilon)$, then by a rearrangement argument, we have
  \begin{equation}\label{2-35}
  \begin{split}
     \mathcal{K}_s\omega_\varepsilon(y) & \le c_s\int_S \frac{\omega_\varepsilon(x')}{|y-x'|^{2-2s}}dx'+C \\
       & \le c_s\int_{B_{R_\sigma\varepsilon}(z_\varepsilon)} \frac{\omega_\varepsilon(x')}{|y-x'|^{2-2s}}dx'+c_s\int_{S\backslash B_{R_\sigma\varepsilon}(z_\varepsilon)} \frac{\omega_\varepsilon(x')}{|y-x'|^{2-2s}}dx'+C         \\
       & \le \frac{c_s}{(\Lambda R_\sigma)^{2-2s}}\frac{1}{\varepsilon^{2-2s}}+\frac{c_s\pi^{1-s}\sigma^s}{s\varepsilon^{2-2s}}+C\\
       & \le \left(\frac{c_s}{(\Lambda R_\sigma)^{2-2s}}+\frac{c_s\pi^{1-s}\sigma^s}{s}\right)\frac{1}{\varepsilon^{2-2s}}+C.
  \end{split}
  \end{equation}
  On the other hand, by Lemma \ref{lem13}, there exists an $\eta_0>0$, independent of $\varepsilon$, such that
  \begin{equation}\label{2-36}
    \mu_\varepsilon\ge \frac{\eta_0}{\varepsilon^{2-2s}},
  \end{equation}
  provided that $\varepsilon$ is small enough. Now, let $\sigma_0$ be such that
  \begin{equation*}
   {c_s\pi^{1-s}\sigma_0^s}/{s}\le\eta_0/3.
  \end{equation*}
  Set
  \begin{equation*}
    \Lambda=\frac{1}{R_{\sigma_0}}\left(\frac{3c_s}{\eta_0} \right)^\frac{1}{2-2s}.
  \end{equation*}
  Combining \eqref{2-35} and \eqref{2-36}, we deduce that
  \begin{equation*}
    \psi_\varepsilon(y)\le \mathcal{K}_s\omega_\varepsilon(y)-\mu_\varepsilon+C\le -\frac{\eta_0}{3}\frac{1}{\varepsilon^{2-2s}}+C<0,
  \end{equation*}
  for all $y\in S\backslash B_{2\Lambda R_{\sigma_0}\varepsilon}(z_\varepsilon)$ if $\varepsilon$ is small enough. It follows that  $\text{supp}(\omega_\varepsilon)\subseteq B_{2\Lambda R_{\sigma_0} \varepsilon}(z_\varepsilon)$. Let $\Lambda_0=4\Lambda R_{\sigma_0}$. By virtue of the $\theta$-symmetrization of $\omega_\varepsilon$ and the constraint $L(\omega_\varepsilon)=1$, it is easy to check that $\text{supp}(\omega_\varepsilon)\subseteq B_{\Lambda_0 \varepsilon}\left((1,0)\right)$. The proof is thus complete.
\end{proof}

\subsection{Steady solutions in the sense of \eqref{1-9}}\label{s7}
 With Lemma \ref{lem14} in hand, we can now show that $\omega_\varepsilon$ is a steady solution in the sense of \eqref{1-9}. More precisely, we have
\begin{lemma}\label{lem15}Let $1/2\le s<1$. Provided that $\varepsilon$ is sufficiently small, it holds
 \begin{equation*}
   \int_{\mathbb{R}^2}\omega_\varepsilon \nabla^\perp(\mathcal{K}_s\omega_\varepsilon+\frac{\alpha_\varepsilon}{2} |x|^2)\cdot\nabla \varphi  dx=0, \ \ \ \forall\,\varphi\in C_0^\infty(\mathbb{R}^2).
 \end{equation*}
\end{lemma}

\begin{proof}
  Let $F_\varepsilon(\tau):=F(\tau)/\varepsilon^2$. It follows from Lemma \ref{lem14} that $F_\varepsilon(\psi_\varepsilon)=0$ on $\partial S$ when $\varepsilon$ is small enough. For any $\varphi\in C_0^\infty(\mathbb{R}^2)$, we can integrate by parts to obtain
\begin{equation*}
  \int_{\mathbb{R}^2}\omega_\varepsilon \nabla^\perp(\mathcal{K}_s\omega_\varepsilon+\frac{\alpha_\varepsilon}{2} |x|^2)\cdot\nabla \varphi dx=-\int_S F_\varepsilon(\psi_\varepsilon)(\partial_{x_2}\partial_{x_1}\varphi-\partial_{x_1}\partial_{x_2}\varphi)dx=0,
\end{equation*}
which completes the proof.
\end{proof}

\subsection{The asymptotic shape of $\omega_\varepsilon$}\label{s8}

Define the (modified) center of $\omega_\varepsilon$ by
\begin{equation*}
  \bar{z}_\varepsilon:=\int_S x\omega_\varepsilon(x)dx.
\end{equation*}
Let us introduce the rescaled version of $\omega_\varepsilon$ as follows:
\begin{equation*}
  \bar{\zeta}_\varepsilon(x)=\varepsilon^2\omega_\varepsilon(\bar{z}_\varepsilon+\varepsilon x),\ \ \ \ x\in \mathbb{R}^2.
\end{equation*}
Then
\begin{equation*}
  0\le \bar{\zeta}_\varepsilon\le 1, \ \ \ \text{supp}(\bar{\zeta}_\varepsilon)\subseteq B_{2\Lambda_0}(0)\ \ \   \text{and}\ \ \ \int_{\mathbb{R}^2}\bar{\zeta}_\varepsilon dx=1.
\end{equation*}
We denote by $\bar{\zeta}^*_\varepsilon$ the symmetric radially nonincreasing Lebesgue-rearrangement of $\bar{\zeta}_\varepsilon$ centered at the origin. The following result determines the asymptotic nature of $\omega_\varepsilon$ in terms of its rescaled version $\bar{\zeta}_\varepsilon$.

\begin{lemma}\label{lem16}
Let $1/2\le s<1$.
Every accumulation point of the family $\{\bar{\zeta}_{\varepsilon}\}_{\varepsilon>0}$ in the weak-star topology of $L^\infty(\mathbb{R}^2)$ must be a radially non-increasing function.
\end{lemma}

\begin{proof}
  Up to a subsequence we may assume that $\bar{\zeta}_\varepsilon\to \bar{\zeta}$ and $\bar{\zeta}^*_\varepsilon\to \bar{\zeta}^*$ weakly-star in $L^\infty(\mathbb{R}^2)$ as $\varepsilon\to 0^+$. By the Riesz rearrangement inequality, we first have
  \begin{equation*}
    c_s\int_{\mathbb{R}^2}\int_{\mathbb{R}^2}\frac{\bar{\zeta}_\varepsilon(x)\bar{\zeta}_\varepsilon(y)}{|x-y|^{2-2s}}dxdy\le  c_s\int_{\mathbb{R}^2}\int_{\mathbb{R}^2}\frac{\bar{\zeta}^*_\varepsilon(x)\bar{\zeta}^*_\varepsilon(y)}{|x-y|^{2-2s}}dxdy.
  \end{equation*}
  Letting $\varepsilon \to 0^+$, we get
  \begin{equation}\label{2-37}
    \mathcal{I}_s(\bar{\zeta})\le \mathcal{I}_s(\bar{\zeta}^*).
  \end{equation}
Let $\bar{\omega}_\varepsilon$ be defined as
\[
{\omega}^*_\varepsilon(x)=\left\{
   \begin{array}{lll}
        \varepsilon^{-2}\bar{\zeta}^*_\varepsilon\big(\varepsilon^{-1}(x-\bar{z}_\varepsilon)\big) &    \text{if} & x\in B_{2\Lambda_0\varepsilon}(\bar{z}_\varepsilon), \\
         0                  &    \text{if} & x\in S\backslash B_{2\Lambda_0\varepsilon}(\bar{z}_\varepsilon).
    \end{array}
   \right.
\]
A direct calculation then yields that as $\varepsilon\to 0^+$,
\begin{equation*}
      \mathcal{E}_\varepsilon({\omega}_\varepsilon)= \frac{\mathcal{I}_s(\bar{\zeta}_\varepsilon)}{2\varepsilon^{2-2s}}-\mathcal{J}_\varepsilon(\bar{\zeta}_\varepsilon)+O(1),
\end{equation*}
and
\begin{equation*}
     \mathcal{E}_\varepsilon({\omega}^*_\varepsilon)= \frac{\mathcal{I}_s(\bar{\zeta}^*_\varepsilon)}{2\varepsilon^{2-2s}}-\mathcal{J}_\varepsilon(\bar{\zeta}^*_\varepsilon)+O(1).
\end{equation*}
Notice that $\mathcal{E}_\varepsilon({\omega}^*_\varepsilon)\le \mathcal{E}_\varepsilon({\omega}_\varepsilon)$ and $\mathcal{J}_\varepsilon(\bar{\zeta}_\varepsilon)=\mathcal{J}_\varepsilon(\bar{\zeta}^*_\varepsilon)$. It follows that  $\mathcal{I}_s(\bar{\zeta}_\varepsilon)\ge \mathcal{I}_s(\bar{\zeta}^*_\varepsilon)+o(1)$ as $\varepsilon \to 0^+$. Letting $\varepsilon \to 0^+$, we get
\begin{equation}\label{2-38}
  \mathcal{I}_s(\bar{\zeta})\ge \mathcal{I}_s(\bar{\zeta}^*).
\end{equation}
Combining \eqref{2-37} and \eqref{2-38}, we conclude that
\begin{equation*}
  \mathcal{I}_s(\bar{\zeta})= \mathcal{I}_s(\bar{\zeta}^*).
\end{equation*}
By Lemma 3.2 in Burchard--Guo \cite{BG}, we know that there exists a translation $\mathcal T$ of $\mathbb{R}^2$ such that $\mathcal T\circ\bar{\zeta}=\bar{\zeta}^*$. Note that
\begin{equation*}
  \int_{\mathbb{R}^2}x\bar{\zeta}(x)dx=\int_{\mathbb{R}^2}x\bar{\zeta}^*(x)dx=0.
\end{equation*}
We must have $\bar{\zeta}=\bar{\zeta}^*$, and the proof is thus complete.
\end{proof}

Now we are ready to prove Theorem \ref{thm1}.

\noindent{\bf Proof of Theorem \ref{thm1}:}
Let
\begin{equation*}
  \omega_{\text{ro},\varepsilon}(x)=\sum_{k=0}^{N-1}\omega_\varepsilon\left(Q_{\frac{2k\pi}{N}}x\right),\ \ \ x\in\mathbb{R}^2.
\end{equation*}
Then the statements of Theorem \ref{thm1} follow from the above lemmas.
\qed

\section{Construction of translating vortex pairs for the gSQG equation}\label{Sec3}
In this section, we provide a variational construction of translating vortex pairs for the gSQG equation. Recall that the kinetic energy of the fluid is given by
\begin{equation*}
  {\text{KE}}_s(\omega):=\frac{1}{2}\int_{\mathbb{R}^2}\int_{\mathbb{R}^2}G_s(x-x')\omega(x)\omega(x')dxdx'.
\end{equation*}
We also introduce the impulse of the fluid as follows:
\begin{equation*}
  \mathcal{P}(\omega)=\int_{\mathbb{R}^2}x_1 \omega(x)dx.
\end{equation*}
For the sake of simplicity, we focus on translating vortex pairs which are symmetric about the $x_2$-axis. More precisely, we first assume
\begin{equation*}
  \omega(x_1,x_2)=-\omega(-x_1,x_2).
\end{equation*}
Let
\begin{equation*}
  \bar{G}_s(x,x')=G_s(x-x')-G_s(\bar{x}-x'),\ \ \text{with}\ \ \bar{x}:=(-x_1, x_2).
\end{equation*}
With the symmetry assumption in hand, we have
\begin{equation*}
  {\text{KE}}_s(\omega)=2\left(\frac{1}{2}\int_{\mathbb{R}^2_+}\int_{\mathbb{R}^2_+}\bar{G}_s(x,x')\omega(x)\omega(x')dxdx'\right)\ \ \ \text{and}\ \ \  \mathcal{P}(\omega)=2\int_{\mathbb{R}^2_+}x_1 \omega(x)dx.
\end{equation*}
Let $W$ be a positive number. Based on this fact, we shall restrict the construction to only one vortex on $\mathbb{R}^2_+$. Set
\begin{equation*}
  \mathcal{E}_{\varepsilon,\text{tr}}(\omega)=\frac{1}{2}\int_{\mathbb{R}^2_+}\int_{\mathbb{R}^2_+}\bar{G}_s(x,x')\omega(x)\omega(x')dxdx'-W\int_{\mathbb{R}^2_+}x_1 \omega(x)dx-\frac{1}{\varepsilon^2}\int_{\mathbb{R}^2_+} J(\varepsilon^2\omega(x))dx.
\end{equation*}
Let
\begin{equation*}
			d=\left(\frac{1}{4\pi W}\frac{\Gamma(2-s)}{\Gamma(s)}\right)^{\frac{1}{3-2s}}, \ b_1=d\mathbf{e}_1,\ b_2=-d\mathbf{e}_1,\ \mathbf{e}_1=(1,0).
		\end{equation*}
 We shall consider the class of admissible function as follows:
\begin{equation*}
  \mathcal{A}_{\text{tr}}:=\Big{\{}\omega\in L^{1+\frac{1}{s}}(\mathbb{R}^2): \omega\ge0,\ \text{supp}(\omega)\subseteq \overline{B_{d/2}(b_1)},\ M(\omega)=1 \Big{\}}.
\end{equation*}
Consider the maximization problems
\begin{equation*}
  \mathcal{C}'_{\varepsilon}:=\sup\Big{\{} \mathcal{E}_{\varepsilon,\text{tr}}(\omega): \omega\in \mathcal{A}_{\text{tr}} \Big{\}}.
\end{equation*}
For a given nonnegative function $\omega$, we shall say that $\omega$ is Steiner symmetric with respect to $x_2$ if
\begin{equation*}
	\omega(x_1,x_2)=\omega(x_1,-x_2) \ \ \ \text{and}\ \ \ \omega \ \text{is non-increasing in}\ x_2\  \text{for}\,x_2>0.	
\end{equation*}
One can check that $\mathcal{E}_{\varepsilon,\text{tr}}$ is increased by the Steiner symmetrization with respect to $x_2$. Let
\begin{equation*}
  \bar{\mathcal{K}}_s\omega=\int_{\mathbb{R}^2_+}\bar{G}_s(x,x')\omega(x')dx'.
\end{equation*}
Using the similar argument as in the proof of Lemma \ref{lem4}, we can obtain the following result.

\begin{lemma}\label{lem17}
Let $0<s<1$ and $0<\varepsilon<\sqrt{\pi}d/2$. Then there exists an $\omega_{\varepsilon,tr}\in \mathcal{A}_{tr}$, which is Steiner symmetric with respect to $x_2$, such that
\begin{equation*}
  \mathcal{E}_{\varepsilon,tr}(\omega_{\varepsilon,tr})=\sup_{\omega\in \mathcal{A}_{tr}}\mathcal{E}_{\varepsilon,tr}(\omega)<+\infty.
\end{equation*}
Moreover, there exist a number $\mu_{\varepsilon,tr}$ such that
\begin{equation}\label{3-1}
		\omega_{\varepsilon,tr}(x)=\frac{1}{\varepsilon^2} f(\psi_{\varepsilon,tr}(x)),\ \ \text{a.e.}\ x\in S,
\end{equation}
with
\begin{equation*}
  \psi_{\varepsilon,tr}(x)=\bar{\mathcal{K}}_s\omega_{\varepsilon,tr}(x)-Wx_1-\mu_{\varepsilon,tr}.
\end{equation*}
\end{lemma}

The maximizing property of $\omega_{\varepsilon,tr}$ will force its support to shrinks to some point in $\overline{B_{d/2}(b_1)}$ as $\varepsilon \to 0$. More precisely, we have
\begin{lemma}\label{lem18}Let $0<s<1$. Then there exists a family of points $\{z_{\varepsilon,tr}\}_{\varepsilon>0}$ on the $x_1$-axis and a positive number $R_0$, such that
  \begin{equation}\label{3-2}
  \text{supp}(\omega_{\varepsilon,tr})\subseteq B_{R_0\varepsilon}(z_{\varepsilon, tr}),
  \end{equation}
  for any $0<\varepsilon<\sqrt{\pi}d/2$. Moreover, $z_{\varepsilon,tr}\to b_1 $ as $\varepsilon \to 0$.
\end{lemma}
\begin{proof}
The proof of \eqref{3-2} is quite similar to that given earlier for Lemma \ref{lem14} and so is omitted. Now, let us prove $z_{\varepsilon,tr}\to b_1 $ as $\varepsilon \to 0$. Up to extracting a subsequence, we assume that $z_{\varepsilon,tr}\to (r_*,0)$ for some $r_*>0$ as $\varepsilon \to 0$. We claim $r_*=d$. Indeed, let
\begin{equation*}
  \tilde{\omega}_{\varepsilon,tr}=\omega_{\varepsilon,tr}(z_{\varepsilon,tr}-b_1+\cdot)\in \mathcal{A}_{\text{tr}}.
\end{equation*}
We derive from $\mathcal{E}_{\varepsilon,tr}(\omega_{\varepsilon,tr})\ge \mathcal{E}_{\varepsilon,tr}(\tilde{\omega}_{\varepsilon,tr})$ that
\begin{equation*}
\begin{split}
   \frac{1}{2}\int_{\mathbb{R}^2_+}\int_{\mathbb{R}^2_+}{G}_s(\bar{x}-x')&\omega_{\varepsilon,tr}(x) \omega_{\varepsilon,tr}(x')dxdx'+W\int_{\mathbb{R}^2_+}x_1 \omega_{\varepsilon,tr}(x)dx    \\
     &  \le\frac{1}{2}\int_{\mathbb{R}^2_+}\int_{\mathbb{R}^2_+}{G}_s(\bar{x}-x')\tilde{\omega}_{\varepsilon,tr}(x)\tilde{\omega}_{\varepsilon,tr}(x')dxdx'+W\int_{\mathbb{R}^2_+}x_1 \tilde{\omega}_{\varepsilon,tr}(x)dx.
\end{split}
\end{equation*}
Let $\varepsilon \to 0$ in the above inequality, we get
\begin{equation*}
  {G}_s(2r_*)/2+Wr_*\le {G}_s(2d)/2+Wd.
\end{equation*}
Since $d$ is the unique minimizer of the function ${G}_s(2\tau)/2+W\tau$ for $\tau>0$, we must have $r_*=d$. The proof is thus complete.
\end{proof}

\begin{lemma}\label{lem19}Let $1/2\le s<1$. Provided that $\varepsilon$ is sufficiently small, it holds
 \begin{equation*}
   \int_{\mathbb{R}^2}\omega_{\varepsilon,tr} \nabla^\perp(\bar{\mathcal{K}}_s\omega_{\varepsilon,tr}-Wx_1)\cdot\nabla \varphi  dx=0, \ \ \ \forall\,\varphi\in C_0^\infty(\mathbb{R}^2).
 \end{equation*}
\end{lemma}
\begin{proof}
  Let $F_\varepsilon(\tau):=F(\tau)/\varepsilon^2$. It follows from Lemma \ref{lem18} that $F_\varepsilon(\psi_\varepsilon)=0$ on $\partial B_{d/2}(b_1)$ when $\varepsilon$ is small enough. For any $\varphi\in C_0^\infty(\mathbb{R}^2)$, we can integrate by parts to obtain
\begin{equation*}
  \int_{\mathbb{R}^2}\omega_{\varepsilon,tr} \nabla^\perp(\bar{\mathcal{K}}_s\omega_{\varepsilon,tr}-Wx_1)\cdot\nabla \varphi  dx=-\int_S F_\varepsilon(\psi_{\varepsilon,tr})(\partial_{x_2}\partial_{x_1}\varphi-\partial_{x_1}\partial_{x_2}\varphi)dx=0,
\end{equation*}
which completes the proof.
\end{proof}

The remainder of the proof is analogous to that in the previous section and so is omitted.

\section{Auxiliary results}

In this appendix, we collect some auxiliary results, which we have been used in the preceding sections. Recall that
\begin{equation*}
  \mathcal{U}_N=\Big\{(r\cos\theta,r\sin\theta)\in \mathbb{R}^2:-\frac{\pi}{N}<\theta<\frac{\pi}{N} \Big\}.
\end{equation*}
Consider a non-negative measurable function $\omega$ defined on $\mathcal{U}_N$. We define its angular Steiner symmetrization $\omega^\sharp$ to be the unique rearrangement of $\omega$ which is Steiner symmetric with respect to $\theta$. In other words, $\omega^\sharp$ is the unique even function for the variable $\theta$ such that
\begin{equation*}
  \omega^\sharp(r,\theta)>\tau\ \ \ \text{if and only if}\ \ \ |\theta|<\frac{1}{2}\,\text{meas}\left\{\theta'\in (-\frac{\pi}{N},\frac{\pi}{N}):\omega(r,\theta')>\tau \right\},
\end{equation*}
for any positive numbers $r$ and $\tau$, and any $-{\pi}/{N}<\theta<{\pi}/{N}$. Using the layer-cake representation of nonnegative measurable functions (see \cite{Lie}), we have
\begin{lemma}\label{A1}
  Let non-negative $\omega\in L^1_{\text{loc}}(\mathcal{U}_N)$ and let $g:\mathbb{R}_+\to \mathbb{R}_+$ be continuous. Then we have
  \begin{equation*}
    \int_{\mathcal{U}_N}g(r)\omega(r,\theta)rdrd\theta=\int_{\mathcal{U}_N}g(r)\omega^\sharp(r,\theta)rdrd\theta,
  \end{equation*}
  provided that these quantities are finite.
\end{lemma}

Recall that
\begin{equation*}
 {E}_s(\omega)=\frac{1}{2}\int_{\mathcal{U}_N}\int_{\mathcal{U}_N}K_s(x,x')\omega(x)\omega(x')dxdx',
\end{equation*}
where the kernel $K_s$ is equal to
\begin{equation*}
  K_s(x,x')=\sum_{k=0}^{N-1}G_s\left(x- Q_{\frac{2k\pi}{N}}x'\right).
\end{equation*}
In the polar coordinates $(r,\theta)$, the kernel $K_s$ can be expressed as
\begin{equation*}
  K_s(r,\theta,r',\theta')=V_s(r,r',\theta-\theta').
\end{equation*}
Given two fixed numbers $1/2<r\neq r'<3/2$, the map $\tau \mapsto V_s(r,r',\tau)$ is smooth on $(-\pi/N, \pi/N)$. Moreover, we have
\begin{itemize}
  \item [(i)]$V_s(r,r',-\tau)=V_s(r,r',\tau)$\ \ \ \text{whenever}\ \ $|\tau|<\pi/N$,
  \item [(ii)]$\partial_\tau V_s(r,r',\tau)<0$\ \ \ \ \ \ \ \ \ \ \ \ \ \ \text{whenever}\ \ $0<\tau<\pi/N$.
\end{itemize}
For proofs, we refer to \cite{Go}. Using these facts, one can easily get the following result which asserts that the energy ${E}_s$ is increased by the angular Steiner symmetrization.

\begin{lemma}\label{A2}
  Suppose non-negative $\omega \in L^\infty(\mathbb{R}^2)$ satisfies $\text{supp}(\omega)\subseteq S$. Then ${E}_s(\omega)\le {E}_s(\omega^\sharp)$.
\end{lemma}

\phantom{s}
 \thispagestyle{empty}

\end{document}